\documentclass[10pt]{article}
\usepackage{amssymb}
\usepackage{amsmath}
\usepackage{amsthm}
\usepackage{cite}
\usepackage{algorithm}
\usepackage{graphicx} 
\usepackage{hyperref}
\usepackage{epstopdf}
\usepackage[english]{babel}
\usepackage{tabularx}
\usepackage{verbatim}
\usepackage{geometry}
\usepackage{enumerate}
\usepackage{caption}
\usepackage{subcaption}
\usepackage{epstopdf}
\usepackage{multirow}
\usepackage{enumitem}
\usepackage{xcolor}
\usepackage{authblk}
\newtheorem{theorem}{Theorem}
\newtheorem{proposition}[theorem]{Proposition}%
\newtheorem{lemma}[theorem]{Lemma}%
\newtheorem{condition}[theorem]{Condition}%
\newtheorem{remark}{Remark}%

\newtheorem{definition}{Definition}%

\def\iPiano{{i$^2$Piano}}

\def\iPianoLA{iPila}

\usepackage{enumitem}

\newlist{algorithmsteps}{enumerate}{1}
\setlist[algorithmsteps,1]{
  label={\textsc{Step}~\arabic*},
  leftmargin=*,
  align=left,
  labelsep=2mm,
}

\def\R{\mathbb R}
\def\bR{\overline{\R}}

\def\N{\mathbb N}

\def\dom{\mathrm{dom}}
\def\argmin{\operatornamewithlimits{argmin}\limits}
\def\prox{{\mathrm{prox}}}

\def\dist{\mathrm{dist}}
\def\kinN{{k\in\N}}
\def\k{^{(k)}}

\def\x{ { x}}

\def\xk{ \x^{(k)}}
\def\yk{ \hat y^{(k)}}

\def\xkk{ \x^{(k+1)}}
\def\xkm{ \x^{(k-1)}}
\def\xkj{ \x^{(k_j)}}

\def\ty{\tilde y}
\def\hy{\hat y}
\def\tyk{ \ty^{(k)}}

\def\hyk{ \hy^{(k)}}
\def\uk{u^{(k)}}
\def\ukm{u^{(k-1)}}
\def\ukj{u^{(k_j)}}

\def\dk{d^{(k)}}
\def\ak{{\alpha_k}}
\def\amin{\alpha_{min}}
\def\amax{\alpha_{max}}

\def\lamk{\lambda_k}

\def\hk{h^{(k)}}

\def\rk{\rho^{(k)}}
\def\rkm{\rho^{(k-1)}}
\def\rkj{\rho^{(k_j)}}
\def\tkk{\tau^{(k+1)}}
\def\tk{\tau^{(k)}}
\def\tkm{\tau^{(k-1)}}

\def\P{\Phi}
\def\F{{\mathcal F}}

\newcommand\silviacorr[1]{#1}

\def\yk{y^{(k)}}

\begin{document}

\title{An abstract convergence framework with application to inertial inexact forward--backward methods}

\author[1]{Silvia Bonettini}

\author[2]{{Peter} {Ochs}} 

\author[1]{{Marco} {Prato}}

\author[3]{{Simone} {Rebegoldi}}

\affil[1]{{Dipartimento di Scienze Fisiche, Informatiche e Matematiche}, {Universit\`a di Modena e Reggio Emilia}, {{Via Campi 213/b}, {Modena}, {41125}, {Italy}}\\
\url{silvia.bonettini@unimore.it} \ \ \url{marco.prato@unimore.it}}

\affil[2]{{Department of Mathematics}, {University of T\"{u}bingen}, {{Auf der Morgenstelle 10}, {T\"{u}bingen}, {72076}, {Germany}}\\ \url{ochs@math.uni-tuebingen.de}
}

\affil[3]{{Dipartimento di Ingegneria Industriale}, {Universit\`a di Firenze}, {{Via di S. Marta 3}, {Firenze}, {50139}, {Italy}}\\ \url{simone.rebegoldi@unifi.it}}

\maketitle
\abstract{In this paper we introduce a novel abstract descent scheme suited for the minimization of proper and lower semicontinuous functions. The proposed abstract scheme generalizes a set of properties that are crucial for the convergence of several first-order methods designed for nonsmooth nonconvex optimization problems. Such properties guarantee the convergence of the full sequence of iterates to a stationary point, if the objective function satisfies the Kurdyka--\L{}ojasiewicz property. The abstract framework allows for the design of new algorithms. We propose two inertial-type algorithms with implementable inexactness criteria for the main iteration update step. The first algorithm, i$^2$Piano, exploits large steps by adjusting a local Lipschitz constant. The second algorithm, iPila, overcomes the main drawback of line-search based methods by enforcing a descent only on a merit function instead of the objective function. Both algorithms have the potential to escape local minimizers (or stationary points)  by leveraging the inertial feature. Moreover, they are proved to enjoy the full convergence guarantees of the abstract descent scheme, which is the best we can expect in such a general nonsmooth nonconvex optimization setup using first-order methods. The efficiency of the proposed algorithms is demonstrated on two exemplary image deblurring problems, where we can appreciate the benefits of performing a linesearch along the descent direction inside an inertial scheme.}




\section{Introduction}\label{sec1}

The design of efficient first-order methods is vital for tackling composite optimization problems of the form
\begin{equation}\label{min_intro}
\min_{x\in\mathbb{R}^n}f(x), \quad f(x)=f_0(x)+f_1(x),
\end{equation}
where $f_1$ is convex and $f_0$ is continuously differentiable on an open set containing the domain of $f_1$. Such problems are frequently encountered in image processing and machine learning applications \cite{Bertero2018,Bottou-Curtis-Nocedal-2018,Chambolle-Pock-2016}, where one of the two terms is usually a data fidelity term and the other one encodes some apriori information on the ground truth \cite{Bertero2018}. Popular and effective first-order methods aiming at solving \eqref{min_intro} include forward--backward (FB) methods \cite{Combettes-Pesquet-2011,Combettes-Wajs-2005,Ochs-etal-2014,Bonettini-Loris-Porta-Prato-Rebegoldi-2017}, whose structure consists in the alternation of a gradient step on $f_0$ followed by a proximal minimization step on $f_1$, block coordinate methods \cite{Bolte-etal-2014,Bonettini-Prato-Rebegoldi-2018,Chouzenoux-etal-2016,Frankel-etal-2015}, Douglas-Rachford methods \cite{Combettes-Pesquet-2011,Li-Pong-2016} and several others.

In recent years, the convergence of first-order descent methods in nonconvex settings has been carefully addressed by relying on the so-called Kurdyka--\L{}ojasiewicz (KL) inequality \cite{Attouch-etal-2013,Bolte-etal-2007,Kurdyka-1998}. This analytical property is satisfied by a large number of objective functions arising in signal processing and machine learning, such as real analytic or semialgebraic functions (see e.g. \cite{Bolte-etal-2007,Bolte-etal-2010b}), { and, more generally, functions that are definable in an $o$-minimal structure \cite{Dries98,Bolte-etal-2007,Bolte-etal-2007b},} thus making quite natural to consider the KL inequality as a standard blanket assumption whenever the objective function is nonconvex. 
The convergence of descent methods under the KL assumption was first considered for { gradient related methods} in \cite{Absil-2005} {  and for proximal methods in \cite{Attouch-Bolte-2009,Attouch-etal-2010}}, where the authors combine the KL inequality with some crucial properties of descent methods to prove the convergence of the iterates to a stationary point of the objective function, under some boundedness assumption. 
In \cite{Attouch-etal-2013}, the authors extend this seminal idea by providing the first {\it abstract descent scheme} in the KL framework, namely a set of abstract properties ensuring the convergence of a generic iterative scheme to a stationary point if combined with the KL inequality. Specific algorithms are then derived from the abstract scheme.\\
This approach, consisting in analyzing and devising new and existing algorithms from an abstract scheme, has been successively adopted by several authors \cite{Bolte-etal-2014,Frankel-etal-2015,Bolte-etal-2018, Ochs-etal-2014,Ochs-2019,Bonettini-Prato-Rebegoldi-2020,Bonettini-Prato-Rebegoldi-2021}. 
In particular, in \cite{Ochs-etal-2014,Ochs-2019}, the authors modify the abstract scheme proposed in \cite{Attouch-etal-2013} in order to include an inertial term inside a classical FB scheme and devise the so-called iPiano (inertial Proximal algorithm for nonconvex optimization) which can be considered as a generalization of the Heavy-Ball method \cite{Polyak-1964,Polyak-1987}.
Similarly, the authors in \cite{Bonettini-Prato-Rebegoldi-2020,Bonettini-Prato-Rebegoldi-2021} adapt the abstract descent scheme to their proposed method VMILA (Variable Metric Inexact Linesearch Algorithm); yet, unlike in \cite{Ochs-etal-2014}, the abstract scheme is designed to include an implementable inexactness criterion for the computation of the proximal point. 

\silviacorr{The aim of this paper is to provide new theoretical and algorithmic tools able for effectively solving problem \eqref{min_intro} in very general settings, i.e., when $f_0$ is not necessarily convex and/or the proximity operator of $f_1$ is not necessarily available in closed form. These difficulties are very common in the framework of inverse problems, for example when the acquisition model is nonlinear or needs to be blindly estimated \cite{Ayers-Dainty-1988}, leading to a nonconvex data fidelity term, and/or the regularization term involves Total-Variation-like functionals or sparsity in nontrivial spaces, so that its proximal operator cannot be computed in closed form or is too costly to compute \cite{Bach-2012,Bonettini-Loris-Porta-Prato-Rebegoldi-2017,Bonettini-Rebegoldi-Ruggiero-2018,Villa-etal-2013}.
}

\silviacorr{More specifically, we give two different types of contributions: (i) From the theoretical point of view, we propose a novel abstract descent scheme for proving the convergence of iterative methods under the KL assumption, which includes as special cases the abstract schemes in \cite{Ochs-2019} and in \cite{Bonettini-Prato-Rebegoldi-2021}, which are currently considered as state-of-the-art. Notably, our proposed abstract scheme provides the theoretical foundation for FB algorithms that include both an inertial term in the iteration rule and an inexact computation of the proximity operator. (ii) In the second part of the paper, we materialize this new theoretical framework into two novel inexact inertial FB algorithms, which improve the modelling flexibility and the practical performance, while providing the same convergence guarantees, compared to other FB-type algorithms. The first algorithm, denominated i$^2$Piano (inertial inexact Proximal algorithm for nonconvex optimization), can be considered as an inexact version of the proximal Heavy-Ball algorithm \cite{Ochs-etal-2014,Ochs-2019} equipped with a backtracking procedure based on a local version of the Descent Lemma. The second one is denoted iPila (inertial Proximal inexact line--search algorithm), which features an inertial--like step followed by a linesearch procedure along the descent direction of a suitable merit function. Such a linesearch strategy allows to compute the inexact proximal point only once per iteration, unlike the backtracking procedure of i$^2$Piano. 
}

The paper is organized as follows. In Section 2 some basic notions on variational analysis and the definition of the KL property are reported. In Section 3 the proposed abstract scheme is presented and its convergence properties analysed. Section \ref{sec:applications} is devoted to the design of the two algorithms i$^2$Piano and iPila. Their analysis within the abstract framework presented in Section 3 is performed in Section \ref{sec:convergence}. Finally, two numerical tests on image deblurring problems are reported in Section \ref{sec:numerical}.

\section{Preliminaries}

In the remainder of the paper, we denote with $\bR=\R\cup \{-\infty,+\infty\}$ the extended real numbers set and $\R^{n\times n}$ the set of $n\times n$ real-valued matrices, while $\|\cdot\|$ denotes the Euclidean norm. 
Given a function $\F:\R^n\to\bR$ and denoting with $\operatorname{dom}(\F)=\{x\in\R^n: \ \F(x)<+\infty\}$ the domain of $\F$, we say that $\F$ is proper if $\operatorname{dom}(\F)\neq\emptyset$ and $\F$ is finite on $\operatorname{dom}(\F)$. The distance operator of a point $x\in\R^n$ to a set $\Omega\subset \R^n$ is defined as
\begin{equation*}
\dist(x,\Omega) = \inf_{y\in\Omega} \|x-y\|.
\end{equation*}
Observe that, if $\Omega = \Omega_1\times\Omega_2$, where $\Omega_1\subset\R^{n_1}$, $\Omega_2\subset\R^{n_2}$, with $n_1+n_2=n$, then for all $x=(x_1,x_2)\in\R^n$, with $x_1\in\R^{n_1}$, $x_2\in\R^{n_2}$, we have
\begin{equation}\label{eq:cartesian}
\dist(x,\Omega) = \sqrt{\dist(x_1,\Omega_1)^2+\dist(x_2,\Omega_2)^2}. 
\end{equation}
This follows by observing that $\|x-z\|^2 = \|x_1-z_1\|^2 + \|x_2-z_2\|^2$ for all $z=(z_1,z_2)\in\Omega$.
\begin{definition}\label{subdiff}
Given a proper, lower semicontinuous function $\F:\R^n\rightarrow \bR$, the \emph{Fr\'{e}chet subdifferential} of $\F$ at $\overline{z}\in\dom(\F)$ is defined as the set \cite[Definition 8.3(a)]{Rockafellar-Wets-1998}
\begin{equation*}
\hat\partial \F(\overline{z}) = \left\{w\in\R^n : \liminf_{u\to \overline{z},u\neq \overline{z}} \frac {\F(u)-\F(\overline{z})-(u-\overline{z})^Tw}{\|u-\overline{z}\|}\geq 0\right\}.
\end{equation*} 
Furthermore, the \emph{limiting subdifferential} of $\mathcal{F}$ at $\overline{z}$ is given by \cite[Definition 8.3(b)]{Rockafellar-Wets-1998}
\begin{equation*}
\partial \mathcal{F}(\overline{z}) = \{w\in\R^n : \exists \  z^{(k)}\to \overline{z}, \ \mathcal{F}(z^{(k)})\to \F(\overline{z}), \ w^{(k)} \in \hat\partial \F(z^{(k)})\to w \ {\text{as}} \ k\to \infty\}.
\end{equation*}
\end{definition}

\begin{definition}\label{critical}
Let $\F:\R^n\rightarrow \bR$ be a proper, lower semicontinuous function. A point $\overline{z}\in\R^n$ is stationary for $\F$ if $0\in\partial \F(\overline{z})$.
\end{definition}

Let us introduce the lazy slope of $\mathcal{F}$ at $\overline{z}$, which is given by \cite[p. 877]{Frankel-etal-2015}
\begin{equation*}
\|\partial \F(\overline{z})\|_- = \inf_{v\in\partial \F(\overline{z})}\|v\|.
\end{equation*}
It is then easy to prove the following sufficient criterion for establishing if a point $\overline{z}\in\R^n$ is stationary for the function $\mathcal{F}$.

\begin{lemma}\cite[Lemma 2.1]{Frankel-etal-2015}\label{suff_criterion}
Let $\F:\R^n\rightarrow \bR$ be a proper, lower semicontinuous function and $\overline{z}\in\R^n$. If there exists $\{z^{(k)}\}_{k\in\N}\subset \R^n$ such that $z^{(k)}\rightarrow \overline{z}$, $\F(z^{(k)})\rightarrow \F(\overline{z})$ and $\displaystyle\liminf_{k\rightarrow \infty}\|\partial \mathcal{F}(z^{(k)})\|_-=0$, then $0\in\partial \F(\overline{z})$.
\end{lemma}

\begin{definition}\label{KL}
Let $\F:\R^n \longrightarrow \bR$  be a proper, lower semicontinuous function. The function $\F$ is said to have the KL property at $\overline{z} \in \dom(\partial \F)$ if there exist $\upsilon \in (0,+\infty]$, a neighborhood $U$ of  $\overline{z}$, a continuous concave function $\phi:[0,\upsilon)\longrightarrow [0,+\infty)$ with $\phi(0)=0$, $\phi\in C^1(0,\upsilon)$, $\phi'(s) > 0$ for all $s \in (0,\upsilon)$, such that the following inequality is satisfied 
\begin{equation*}
\phi'(\F(z)-\F(\overline{z})) \|\partial \F(z)\|_- \geq 1
\end{equation*}
for all $z \in U \cap \{z\in\R^n: \F(\overline{z}) < \F(z) < \F(\overline{z}) + \upsilon\}$.\\
If $\F$ satisfies the KL property at each point of $\dom(\F)$, then $\F$ is called a KL function.
\end{definition}
{  The KL property is always satisfied around non stationary points \cite[Remark 4(b)]{Attouch-etal-2010}, whereas it might fail to hold around stationary points. If we assume $\mathcal{F}$ is continuously differentiable with $\mathcal{F}(\overline{z})=0$, then the KL property can be rewritten as 
\begin{equation}\label{eq:ine_KL}
    \|\nabla(\phi \circ \mathcal{F})(z)\|\geq 1
\end{equation} 
{ around the stationary point $\overline{z}$, which means that the values of $\mathcal F$ can be reparametrized by a so-called \emph{desingularization function} $\phi$ such that a ``singular region'', i.e., a region in which the gradients are arbitrarily small, turns into a ``regular region'', i.e., a neighbourhood of $\overline{z}$ where the gradients are bounded away from zero \cite{Attouch-etal-2013}.}
}

In \cite[Lemma 6]{Bolte-etal-2014}, the following \emph{uniformized} version of the KL property is introduced, where the KL inequality holds with the same desingularization function for all points in a suitable neighborhood of a compact set where the function is constant.
\begin{lemma}\label{lemma:UKL} Let $\F:\R^n\to \bR$ be a proper, lower semicontinuous function and $X\subset \R^n$ a compact set. Suppose that $\F$ satisfies the KL property at each point belonging to $X$ and that $\F$ is constant over $X$, i.e., $\F(\bar x) = \bar \F\in \R$ for all $\bar x\in X$. Then, there exists $\mu,\upsilon>0$ and a function $\phi$ as in Definition \ref{KL} such that 
\begin{equation}\label{UKL}
\phi'(\F(z)-\bar \F) \|\partial \F(z)\|_- \geq 1, \quad  \forall z\in \bar B 
\end{equation} 
where the set $\bar B$ is defined as
\begin{equation}\label{UKLbound}
\bar B = \{z\in\R^n: \dist(z,X)<\mu \mbox{ and } \bar \F < \F(z) < \bar \F + \upsilon\}.
\end{equation}
\end{lemma}

\section{Abstract algorithm scheme}\label{sec:abstract}
In the following, we are interested in proving the convergence of an abstract descent algorithm to a stationary point of a proper, lower semicontinuous function $\F$. Such abstract algorithm is defined through a specific set of properties that are shared by several first-order methods designed for nonsmooth nonconvex optimization, including gradient descent methods \cite{Absil-2005}, forward--backward methods \cite{Bonettini-Loris-Porta-Prato-Rebegoldi-2017,Chouzenoux-etal-2014,Ochs-etal-2014} and block coordinate methods \cite{Attouch-etal-2013,Bonettini-Prato-Rebegoldi-2018,Frankel-etal-2015}. Similarly to other abstract descent algorithms in the KL framework, the two main ingredients guaranteeing the convergence of our scheme are the {\it sufficient decrease condition} and the {\it relative error condition}, the latter being related to the minimization subproblem that one has to (inexactly) solve at each iteration of a first-order method. However, unlike in previous works in the literature, we require that the relative error condition is satisfied at a point that might be different from the actual iterate generated by the method. As we will see in Section \ref{sec:applications}, this simple modification allows to circumvent the issue of the actual implementation of the relative error condition, \silviacorr{which seems rather difficult to impose in practice \cite{Bonettini-Prato-Rebegoldi-2020,Noll-2014}}, allowing to include inexact forward-backward methods equipped with an implementable inexactness criterion for the solution of the minimization subproblem.
\def\tk{ s^{(k)}}
\def\tkk{ s^{(k+1)}}
\def\tkm{ s^{(k-1)}}
\begin{condition}[Abstract algorithm scheme]\label{definition:abstract} Let $\F:\R^{n}\times\R^m\to \bR$ be a proper, lower semicontinuous function and $\P:\R^n\times \R^q\to\bR$ a proper, lower semicontinuous, bounded from below function. Consider two sequences $\{\xk\}_{k\in\N}$, $\{\uk\}_{\kinN}$ in $\R^n$, a sequence $\{\rk\}_\kinN$ in $\R^m$, a sequence $\{\tk\}_{\kinN}$ in $\R^q$ and a sequence of nonnegative real numbers $\{d_k\}_\kinN$ such that the following relations are satisfied.
\begin{enumerate}[label={\it[H\arabic*]}]
\item \label{H1}There exists a sequence of positive real numbers $\{a_k\}_{k\in\N}$ such that 
\begin{equation*}
\P(\xkk,\tkk)+a_k d_k^2 \leq \P(\xk,\tk), \quad \forall \ k\geq 0.
\end{equation*}
\item \label{H2} There exists a sequence of nonnegative real numbers $\{r_k\}_\kinN$ with $\lim\limits_{k\to\infty} r_k = 0$ such that
\begin{equation*}
\P(\xkk,\tkk)\leq \F(\uk,\rk)\leq \P(\xk,\tk)+r_k, \quad \forall \ k\geq 0.
\end{equation*}
\item \label{H3} There exist $b>0$, a sequence of positive real numbers $\{b_k\}_{k\in\N}$, a summable sequence of nonnegative real numbers $\{\zeta_k\}_{k\in\N}$, a non-empty finite index set $I\subset \mathbb{Z}$ and $\theta_i\geq 0$, $i\in I$ with $\sum_{i\in I} \theta_i=1$ such that, setting $d_j=0$ for $j\leq 0$, we have
\begin{equation*}
b_{k+1}\|\partial \F(\uk,\rk)\|_-\leq b \sum_{i\in I}\theta_id_{k+1-i} + \zeta_{k+1}, \quad \forall \ k\geq0.
\end{equation*}
\item\label{H6} If $\{(x^{(k_j)},\rkj)\}_{j\in\N}$ is a subsequence of $\{(\xk,\rk)\}_\kinN$ converging to some $ (x^*,\rho^*)\in\R^n\times \R^m$, then we have for $\{u^{(k_j)}\}_{j\in\mathbb{N}}$:
\begin{equation*}
\lim_{j\to\infty}\|\ukj-x^{(k_j)}\|=0, \quad \lim_{j\to\infty} \F(\ukj,\rkj) = \F(x^*,\rho^*).
\end{equation*}
\item \label{H4} There exists a positive real number $p>0$ and $k'\in\mathbb{Z}$ such that
\begin{equation*}
\|\xkk-\xk\|\leq p d_{k+k'}, \quad \forall \ k\geq 0.
\end{equation*}
\item \label{H7} The sequences $\{a_k\}_{k\in\N}$, $\{b_k\}_{k\in\N}$ satisfy the following conditions
$$
\sum_{k=0}^{+\infty}b_k=+\infty, \quad \sup_{k\in\N}\frac{1}{b_ka_k}<+\infty, \quad \inf_{k\in\N}a_k>0. 
$$
\end{enumerate}
\end{condition}
The abstract scheme given in Conditions \ref{definition:abstract} can be seen as a further extension of the one proposed in \cite{Ochs-2019}, which is indeed recovered by setting $\P=\F$, $\uk = \xkk$ and $\rk\equiv s^{(k)}$. In the following, we discuss in detail conditions \ref{H1}-\ref{H7} and their relation with the abstract scheme in \cite{Ochs-2019}.

\begin{itemize}
\item { Condition \ref{H1} is the {\it sufficient decrease condition}, which is imposed on the proper, lower semicontinuous function $\Phi$}. The quantity $a_kd_k^2$ measures the amount of the decrease. Note that, in earlier works based on the KL property \cite{Attouch-etal-2013,Bolte-etal-2014,Bonettini-Loris-Porta-Prato-Rebegoldi-2017,Frankel-etal-2015}, condition \ref{H1} is usually presented by setting $d_k=\|x^{(k+1)}-\xk\|_2$, $\tk\equiv 0$ and $\Phi(x,s)=f(x)$, being $f$ the function to minimize. { The generalized condition reported here is the same one introduced in the recent work \cite{Ochs-2019}, although here it is imposed on a function $\Phi$ that is different from the function $\F$ appearing in condition \ref{H3}, whereas in \cite{Ochs-2019} the two functions were identical.
}

\item Condition \ref{H2} is crucial in the convergence proof of Theorem \ref{thm:convergence}. Indeed it ensures that the sequence $\{\F(\uk,\rk)\}_{k\in\N}$ converges to a limit value $\F^*\in\R$, thus allowing to apply the uniformized KL property at the point $(\uk,\rk)$ for all sufficiently large $k$ and, furthermore, it enables the combination of \ref{H3} and \ref{H1} with the KL inequality. Imposing condition \ref{H2} is required only when $\Phi\neq \F$, this is why it does not appear in \cite{Ochs-2019}. 

\item Condition \ref{H3} is the so-called {\it relative error condition}, which is related to the (possibly) inexact solution of the minimization subproblem performed at each iteration of a first-order method. In the previous literature \cite{Attouch-etal-2013,Bolte-etal-2014,Bonettini-Loris-Porta-Prato-Rebegoldi-2017,Frankel-etal-2015}, such condition is usually employed by setting $u^{(k)}=x^{(k+1)}$, $d_k=\|x^{(k+1)}-\xk\|_2$, $\rk \equiv 0$, $I=\{1\}$, $\theta_1=1$ and $\F(u,\rho)=f(u)$, being $f$ the function to minimize. In \cite{Ochs-2019}, a general positive term $d_k$, a finite index set $I$, a variable parameter $\tk$ and a generic { merit} function $\F$ are employed, while keeping $\uk=x^{(k+1)}$ and $\rk\equiv 0$. Here we also allow the sequence $\{\uk\}_{k\in\N}$ to be distinct from $\{x^{(k)}\}_{k\in\N}$ and the parameters $\{\rk\}_{k\in\N}$ to vary at each iteration. The reason to do so comes from the fact that condition \ref{H3} is hard to enforce algorithmically on $x^{(k+1)}$ when the minimization subproblem is solved inexactly, as noted in \cite{Bonettini-Prato-Rebegoldi-2020,Bonettini-Prato-Rebegoldi-2021,Noll-2014}. However, if a specific, implementable inexactness criterion is adopted for the solution of the subproblem, then the same condition holds for a { merit} function $\F$ evaluated at a different iterate $(u^{(k)},\rk)$. { For instance, this is observed in the convergence analysis of VMILA algorithm} \cite{Bonettini-Loris-Porta-Prato-Rebegoldi-2017,Bonettini-Prato-Rebegoldi-2020,Bonettini-Prato-Rebegoldi-2021}. In \cite{Bonettini-Loris-Porta-Prato-Rebegoldi-2017}, VMILA is included in the KL framework by setting $\Phi(x,s)=\F(x,\rho)=f(x)$, $\uk=x^{(k+1)}$, $d_k=\|x^{(k+1)}-\xk\|_2$, $\rk \equiv \tk\equiv 0$ and noting that, in so doing, the relative error condition holds only by exactly computing the proximal operator. In \cite{Bonettini-Prato-Rebegoldi-2020}, the authors include VMILA in the abstract scheme in a different way, by using $\Phi(x,s)=f(x)$, a { merit} function $\F$ defined upon the concept of forward--backward envelope of $f$ \cite{Stella-etal-2017}, the iterate $\uk$ as the inexact proximal-gradient point $\tyk$, $d_k=\|x^{(k+1)}-\xk\|_2$ and $\rk$ as the error parameter due to the computation of $\tyk$. Finally, in \cite{Bonettini-Prato-Rebegoldi-2021}, VMILA is framed by setting $\Phi(x,s)=f(x)$, $\F(u,\rho)=f(u)+\rho^2/2$, the iterate $\uk$ as the exact proximal--gradient point $\yk$ and $d_k=-h^{(k)}(\tyk)$, where $h^{(k)}$ is the function to minimize when computing the approximation $\tyk$ of the exact point $\yk$. In other words, from the analysis in \cite{Bonettini-Prato-Rebegoldi-2021}, it turns out that we are able to enforce the relative error condition at the exact point $\yk$, which we do not need to compute explicitly, provided that the approximation $\tyk$ is computed using a specific criterion.
%

\item Condition \ref{H6} is the analogue of the so-called {\it continuity condition} in \cite{Ochs-2019}. Here we impose the property for all converging subsequences $(x^{(k_j)},\rho^{(k_j)})$, whereas in \cite{Ochs-2019} it is only required the existence of one such subsequence. This is because, unlike in \cite{Ochs-2019}, we need to ensure that the distance between $\{(\uk,\rk)\}_{k\in\N}$ and the limit set of $\{(\xk,\rk)\}_{k\in\N}$ converges to $0$ (see Lemma \ref{lemma:comp2}\ref{lemcompiv2}).

\item Condition \ref{H4} is also called the {\it distance condition}. It states the connection between the general term $d_k$ and the Euclidean norm, which is fundamental in order to prove the finite length of the sequence $\{\xk\}_{k\in\N}$. Note that the distance condition given in \cite{Ochs-2019} is slightly more general than \ref{H4}; however, that condition alone allows to prove only the finite length of the sequence $\{d_k\}_{k\in\N}$, which in general does not imply the convergence of the sequence $\{\xk\}_{k\in\N}$. In order to obtain the strongest result, condition \ref{H4} is then imposed in \cite[Theorem 10]{Ochs-2019}.

\item Condition \ref{H7} is the same as the {\it parameter condition} in \cite{Ochs-2019}. These requirements on the sequences $\{a_k\}_{k\in\N}$, $\{b_k\}_{k\in\N}$ were first introduced in \cite{Frankel-etal-2015} in order to generalize the abstract descent scheme in \cite{Attouch-etal-2013}.
\end{itemize}

In the remainder of this section, we will denote with $\{\xk\}_{k\in\N}$, $\{\uk\}_{\kinN}$, $\{\rk\}_{k\in\N}$, $\{\tk\}_{k\in\N}$ the sequences complying with Conditions \ref{definition:abstract}. Furthermore, let us define the set of all limit points of the sequence $\{(\xk,\rk)\}_\kinN$:
\begin{equation*}
\Omega^*(x^{(0)},\rho^{(0)}) = \{(x^*,\rho^*)\in\R^n\times \R^m: \exists \ \{k_j\}_{j\in\N}\subset \N \mbox{ such that } (\xkj,\rkj) \rightarrow (x^*,\rho^*)\}.
\end{equation*}
Note that the set $\Omega^*(x^{(0)},\rho^{(0)})$ can be written as
\begin{equation*}
\Omega^*(x^{(0)},\rho^{(0)}) = X^*(x^{(0)})\times R^*(\rho^{(0)})
\end{equation*}
where $X^*(x^{(0)})=\{x^*\in\R^n: \exists \ \{k_j\}_{j\in\N}\subset \N \mbox{ such that } \xkj \rightarrow x^*\}\subset \R^n$, $R^*(\rho^{(0)})=\{\rho^*\in\R^m: \exists \ \{k_j\}_{j\in\N}\subset \N \mbox{ such that } \rkj \rightarrow \rho^*\}\subset \R^m$.
\begin{lemma}\label{lemma:comp2} Let Conditions \ref{definition:abstract} be satisfied. Suppose that $\{(\xk,\rk)\}_\kinN$ is a bounded sequence. Then the following facts hold true.
\begin{enumerate}[label={\it(\roman*)}]
\item\label{lemcompi2} $\Omega^*(x^{(0)},\rho^{(0)})$ is nonempty and compact.
\item\label{lemcompii2} There exists $\F^*\in\R$ such that $\displaystyle\lim_{k\to\infty} \P(\xk,\tk)=\lim_{k\to\infty} \F(\uk,\rk) = \F^*$.
\item\label{lemcompiv2} We have
\begin{equation*}
\lim_{k\to\infty }\dist((\xk,\rk),\Omega^*(\x^{(0)},\rho^{(0)})) = \lim_{k\to\infty }\dist((\uk,\rk),\Omega^*(\x^{(0)},\rho^{(0)})) = 0.
\end{equation*}
\item\label{lemcompiii2} We have $\F(x^*,\rho^*)=\F^*$, $\forall \ (x^*,\rho^*)\in \Omega^*(x^{(0)},\rho^{(0)})$. 
\end{enumerate}
\end{lemma}
\begin{proof}
(i) Since the sequence $\{(\xk,\rk)\}_\kinN$ is bounded, it admits at least a limit point and, hence, $\Omega^*(x^{(0)},\rho^{(0)})$ is nonempty. Compactness is proved in \cite[Lemma 5]{Bolte-etal-2014}.\\
(ii) From \ref{H1} we have that the sequence $\{\P(\xk,\tk)\}_\kinN$ is nonincreasing and, since $\P$ is bounded from below, there exists $\F^*\in \R$ such that
\begin{equation*}
\lim_{k\to\infty}\P(\xk,\tk)=\F^*.
\end{equation*}
The previous relation combined with \ref{H2} proves Part \ref{lemcompii2}.\\
(iii) Since $\{(\xk,\rk)\}_{\kinN}$ is bounded and by definition of $\Omega^*(x^{(0)},\rho^{(0)})$, we have
$$\lim_{k\to\infty }\dist((\xk,\rk),\Omega^*(\x^{(0)},\rho^{(0)})) =0.$$
Observing that $\Omega^*(x^{(0)},\rho^{(0)})=X^*(\x^{(0)})\times R^*(\rho^{(0)})$ and recalling \eqref{eq:cartesian}, the previous limit implies
\begin{equation*}
\lim_{k\to\infty }\dist(\xk,X^*(\x^{(0)}))= 0, \ \ \ \ \lim_{k\to\infty }\dist(\rk,R^*(\rho^{(0)})) =0.
\end{equation*} 
Combining the boundedness of $\{\xk\}_{\kinN}$ with the definition of $X^*(\x^{(0)})$ and property \ref{H6}, we obtain $\displaystyle\lim_{k\to\infty }\dist(\uk,X^*(\x^{(0)})) =0$, which together with the second limit above yields
\begin{equation*}
\lim_{k\to\infty }\dist((\uk,\rk),\Omega^*(\x^{(0)},\rho^{(0)})) = 0.
\end{equation*} 
(iv) This point follows directly from part \ref{lemcompii2} and \ref{H6}.
\end{proof}
\begin{theorem}\label{thm:convergence} Let Conditions \ref{definition:abstract} be satisfied, and suppose that $\{(\xk,\rk)\}_\kinN$ is a bounded sequence and that $\F$ is a KL function. Then the following statements are true.
\begin{itemize}
\item[(i)] The sequence $\{d_k\}_{k\in\N}$ is summable, i.e., it satisfies
$$
\sum_{k=0}^{+\infty}d_k <+\infty.
$$
\item[(ii)] The sequence $\{\xk\}_{k\in\N}$ has finite length, i.e., it satisfies
$$
\sum_{k=0}^{+\infty}\|\xkk-\xk\| <+\infty
$$ 
and thus $\{\xk\}_{k\in\N}$ is a convergent sequence.
\item[(iii)] If also $\{\rk\}_{k\in\N}$ converges, then the sequence $\{(\xk,\rk)\}_{k\in\N}$ converges to a stationary point for $\F$. 
\end{itemize}
\end{theorem}
\begin{proof}
(i) By Lemma \ref{lemma:comp2}\ref{lemcompi2}-\ref{lemcompiii2}, the function $\F$ is constant over the compact set $\Omega^*(x^{(0)},\rho^{(0)})$, therefore we can apply Lemma \ref{lemma:UKL}. Let $\upsilon,\mu,\phi,\bar B$ be as in Lemma \ref{lemma:UKL}. Thanks to Lemma \ref{lemma:comp2}\ref{lemcompii2}-\ref{lemcompiv2} and \ref{H2}, there exists a positive integer $k_0$ such that  
\begin{equation}\label{upsilon}
\P(\xk,\tk)+r_k < \F^* + \upsilon, \quad \dist( (\uk,\rk),\Omega^*(x^{(0)},\rho^{(0)}) ) < \mu
\end{equation}
for all $k\geq k_0$.
Without loss of generality, up to a translation of the iteration index, we can assume $k_0=0$.\\
Let us now set $c=\sup_k 1/(a_kb_k)$, where $c<+\infty$ due to \ref{H7}, $\zeta_k'=\zeta_k/b$ and 
\begin{equation*}
\phi_k = \frac{b}{c}(\phi(\P(\xk,\tk)-\F^*)- \phi(\P(\xkk,\tkk)-\F^*))
\end{equation*}
and prove that
\begin{equation}\label{th2:eq1}
2d_k\leq \phi_k+\sum_{i\in I}\theta_id_{k-i}+\zeta_k', \quad \forall \ k\geq 1.
\end{equation}
We first observe that the definition of $\phi_k$ is well posed, since $\F^*$ is the limit of the nonincreasing sequence $\{\P(\xk,\tk)\}_\kinN$ (see Lemma \ref{lemma:comp2}) and, hence, we have $\P(\xk,\tk)\geq \F^*$, $\forall  \ k\in\N$. Moreover, the monotonicity of both function $\phi$ and sequence $\{\P(\xk,\tk)\}_\kinN$ implies that $\phi_k\geq 0$, $\forall  \ k\in\N$.\\
Let us now consider the two cases $d_k=0$ and $d_k>0$ separately.\\
If $d_k=0$, inequality \eqref{th2:eq1} holds trivially. Otherwise, for any iteration index $k\geq 1$ such that $d_k>0$, taking into account \ref{H1} and \ref{H2}, we can write
\begin{equation*}
\F^*\leq\P(\xkk,\tkk)<\P(\xk,\tk)\leq \F(\ukm,\rkm).
\end{equation*}
On the other hand, the rightmost inequality in \ref{H2} gives
\begin{equation*}
\F(\ukm,\rkm)\leq \P(\xkm,\tkm) + r_{k-1},
\end{equation*}
which, in view of \eqref{upsilon}, implies that $(\ukm,\rkm)\in \bar B$. 
Then, we can write the KL inequality related to the function $\F$ at $(\ukm,\rkm)$:
\begin{equation*}
\phi'(\F(\ukm,\rkm)-\F^*)\geq \frac{1}{\|\partial \F(\ukm,\rkm)\|_-}.
\end{equation*}
Furthermore, combining the previous inequality with \ref{H3} yields
\begin{equation*}
\phi'(\F(\ukm,\rkm)-\F^*)\geq \frac{1}{\frac{b}{b_k}\sum\limits_{i\in I}\theta_id_{k-i}+\frac{1}{b_k}\zeta_k}.
\end{equation*}
Since $\phi$ is concave, $\phi'$ is nonincreasing. Therefore, \ref{H2} implies
\begin{equation}\nonumber
\phi'(\P(\xk,\tk)-\F^*)\geq \phi'(\F(\ukm,\rkm)-\F^*).
\end{equation}
Exploiting again the concavity of $\phi$, we have
\begin{align*}
\phi(\P(\xk,\tk)-\F^*) &- \phi(\P(\xkk,\tkk) -\F^*)\geq \\
& \phi'(\P(\xk,\tk)-\F^*)(\P(\xk,\tk)-\P(\xkk,\tkk)).
\end{align*}
Combining the last three relations with \ref{H1} leads to
\begin{eqnarray*}
\phi(\P(\xk,\tk)-\F^*)-\phi(\P(\xkk,\tkk)-\F^*)&\geq & \frac{\P(\xk,\tk)-\P(\xkk,\tkk)}{\frac{b}{b_k}\sum\limits_{i\in I}\theta_id_{k-i}+\frac{1}{b_k}\zeta_k}\\
&\geq & \frac{a_kd_k^2}{\frac{b}{b_k}\sum\limits_{i\in I}\theta_id_{k-i}+\frac{1}{b_k}\zeta_k}.
\end{eqnarray*}
Recalling the definition of $\phi_k$ and $\zeta_k'$, the above inequality implies the following one
\begin{equation}\nonumber 
d_k^2\leq \phi_k\left(\sum_{i\in I}\theta_id_{k-i}+\zeta_k'\right).
\end{equation}
Taking the square root of both sides and using the inequality $2\sqrt{uv}\leq u+v$ on the right-hand-side, we obtain
\eqref{th2:eq1}.\\
Summing \eqref{th2:eq1} from 1 to $k$ leads to
\begin{equation}\label{th2:eq2}
2 \sum_{j=1}^k d_j \leq \sum_{j=1}^k\phi_j+\sum_{j=1}^k \sum_{i\in I}\theta_i d_{j-i} + \sum_{j=1}^k\zeta'_j.
\end{equation}
We now observe that 
\begin{eqnarray*}
\sum_{j=1}^k\phi_j &=& \frac b c (\phi(\P(x^{(1)},s^{(1)})-\F^*)-\phi(\P(x^{(k+1)},\tkk)-\F^*))\\
 &\leq& \frac b c \phi(\P(x^{(1)},s^{(1)})-\F^*),
\end{eqnarray*}
where the rightmost inequality follows from the positive sign of $\phi$. Furthermore, the second sum in the right-hand side of \eqref{th2:eq2} can be rewritten as below
\begin{align*}
\sum_{j=1}^k\sum_{i\in I}\theta_i d_{j-i}&=\sum_{i\in I}\sum_{j=1}^k\theta_id_{j-i} =\sum_{i\in I}\sum_{r=1-i}^{k-i}\theta_id_{r}\\
&\leq \sum_{i\in I}\sum_{r=1-i}^{0}\theta_id_{r}+\left(\sum_{i\in I}\theta_i\right)\sum_{r=1}^{k}d_{r}+\sum_{i\in I}\sum_{r=k+1}^{k-i}\theta_id_{r}\\
&=\sum_{i\in I}\sum_{r=1-i}^{0}\theta_id_{r}+\sum_{r=1}^{k}d_{r}+\sum_{i\in I}\sum_{r=k+1}^{k-i}\theta_id_{r}
\end{align*}
where we have used the change of variable $r=j-1$ and the property $\sum_{i\in I}\theta_i=1$. Note that the sums appearing in the previous relation are assumed to be zero whenever the start index of the summation is larger than the termination index. Therefore we can write
\begin{equation*}
2 \sum_{j=1}^k d_j \leq \sum_{i\in I}\sum_{r=1-i}^{0}\theta_id_{r}+\sum_{r=1}^{k}d_{r}+\sum_{i\in I}\sum_{r=k+1}^{k-i}\theta_id_{r} + \frac b c \phi(\P(x^{(1)},s^{(1)})-\F^*)+\sum_{j=1}^k\zeta'_j
\end{equation*}
which clearly implies
\begin{equation}\label{eq:crucial}
\sum_{j=1}^k d_j \leq \sum_{i\in I}\sum_{r=1-i}^{0}\theta_id_{r}+\sum_{i\in I}\sum_{r=k+1}^{k-i}\theta_id_{r} + \frac b c \phi(\P(x^{(1)},s^{(1)})-\F^*)+\sum_{j=1}^k\zeta'_j.
\end{equation}
At this point, observe that the first two sums in the right-hand side of \eqref{eq:crucial} are finite linear combinations of the terms $\{d_r\}_{r\in\N}$. Conditions \ref{H1} and \ref{H7} ensure that $d_k\rightarrow 0$, hence those sums are converging to $0$ for $k\to \infty$. Noting also that $\{\zeta'_k\}_{k\in\N}$ is summable and taking the limit of \eqref{eq:crucial} for $k\to \infty$, we obtain  
\begin{equation}\label{eq:limd}
\sum_{k=0}^{\infty} d_k < +\infty.
\end{equation}
(ii) Combining \eqref{eq:limd} with \ref{H4}, we also obtain
\begin{equation*}
\sum_{k=0}^{\infty}\|\xkk-\xk\|\leq p \sum_{k=0}^{\infty} d_{k+k'} < +\infty.
\end{equation*}
which implies that the sequence $\{\xk\}_\kinN$ converges to a point $x^*\in\R^n$.\\
(iii) Let $(x^*,\rho^*)\in\R^n\times\R^m$ be the unique limit point of the sequence $\{(\xk,\rk)\}_{k\in\N}$, namely $(x^{(k)},\rho^{(k)})\rightarrow (x^*,\rho^*)$. By using \ref{H6}, it follows that $(u^{(k)},\rho^{(k)})\rightarrow (x^*,\rho^*)$ and $\F(u^{(k)},\rho^{(k)})\rightarrow \F(x^*,\rho^*)$. Furthermore, summing \ref{H3} for $k=0,\ldots,K$ yields
$$
\sum_{k=0}^{K}b_{k+1}\|\partial \F(\uk,\rk)\|_{-}\leq b\sum_{k=0}^{K}\sum_{i\in I}\theta_i d_{k+1-i}+\sum_{k=0}^{K}\zeta_k.
$$
Taking the limit for $K\rightarrow \infty$, using \eqref{eq:limd} and recalling that $\{\zeta_k\}_{k\in\N}$ is summable, we obtain
\begin{equation*}
\sum_{k=0}^{\infty} b_{k+1}\|\partial \F(\uk,\rk)\|_{-}<+\infty.
\end{equation*}
Since \ref{H7} requires $\sum_{k}b_k=+\infty$, the previous relation implies that
$$
\liminf_{k\rightarrow \infty}\|\partial \F(\uk,\rk)\|_{-}=0.
$$
In conclusion, the sequence $\{(u^{(k)},\rho^{(k)})\}_{k\in\N}$ satisfies all the hypotheses of Lemma \ref{suff_criterion}, which means that $0\in\partial \mathcal{F}(x^*,\rho^*)$.

\end{proof}

\section{Applications of the abstract scheme}\label{sec:applications}
In this section, we show how we can devise some brand new forward--backward--type algorithms satisfying Conditions \ref{definition:abstract} and, hence, guarantee their convergence to a stationary point in virtue of Theorem \ref{thm:convergence}. In particular, from now on, we address the problem
\begin{equation}\label{composite_problem}
\min_{x\in\R^n}f(x), \ \ \ f(x) =  f_0(x) + f_1(x),
\end{equation}
where we assume that $f_0,f_1$ are as follows:
\begin{enumerate}[label=\it{[A\arabic*]}]
\item\label{assi} $f_1:\R^n\to \overline{\R}$ is a proper, lower semicontinuous, convex function;
\item\label{assii} $f_0:\R^n\to\R$ is continuously differentiable on an open set $\Omega_0\supset \overline{\dom(f_1)}$;
\item\label{assiii} $f_0$ has $L-$Lipschitz continuous gradient on $\dom(f_1)$, i.e.,
\begin{equation}\nonumber
\|\nabla f_0(x)-\nabla f_0(y)\|\leq L\|x-y\|, \ \ \forall x,y\in\dom(f_1),
\end{equation}
for some $L>0$.
\item\label{assiv} $f$ is bounded from below.
\end{enumerate}
Under the above assumptions, for any $z\in\dom(f_1)$, the following subdifferential calculus rules hold \cite[Proposition 8.12, Exercise 8.8(c)]{Rockafellar-Wets-1998}
\begin{align}
\partial f_1(z)&=\{w\in\mathbb{R}^n: f_1(y)\geq f_1(z)+\langle w,y-z\rangle, \ \forall y\in\mathbb{R}^n\}\nonumber\\
\partial f(z) &= \{\nabla f_0(z)\} + \partial f_1(z).\label{sum:subdiff}
\end{align}

\silviacorr{One of the most popular strategies for accelerating first order methods consists in including an inertial, or heavy ball, term to the iteration rule.  This idea was originally proposed in \cite{Polyak-1964} for the gradient descent method, and consists in combining the gradient direction and the direction obtained from the last two iterates. This approach has been further developed in the seminal papers \cite{Nesterov-2005,Beck-Teboulle-2009b} and successively realized in a variety of algorithms. In this paper, in the general nonsmooth nonconvex setup of \eqref{composite_problem}, we consider the therefrom inspired inertial proximal-gradient method proposed in \cite{Ochs-etal-2014,Ochs-2019}, which is defined by the iteration}
\begin{equation}\label{eq:inertial_FB_iterate}
\xkk = \prox_{\alpha_k f_1}(\xk - \alpha_k\nabla f_0(\xk) +\beta_k(\xk-\xkm)),
\end{equation}
where $\alpha_k,\beta_k$ are suitably chosen parameters. By definition of the proximity operator, the above updating rule consists in defining the new point $\xkk$ as the unique solution of the minimization problem
\begin{equation}\label{eq:mini_sub_prob}
\min_{y\in \R^n} f_1(y) -f_1(\xk) + \langle \nabla f_0(\xk)-\frac{\beta_k}{\alpha_k}(\xk-\xkm),y-\xk\rangle + \frac{1}{2\alpha_k}\|y - \xk\|^2 .
\end{equation}
We will refer to the minimizer of this problem as the \emph{inertial proximal gradient point}. If $\beta_k=0$, we recover the standard proximal gradient point, otherwise the inertial step $\xk-\xkm$ is included in the argument of the proximal gradient operator, with the aim of improving the convergence behaviour of the overall method.

However, in several practical situations the exact minimization of \eqref{eq:mini_sub_prob} is infeasible, while high precision approximations can be computed efficiently. Therefore, in the following we address the key challenge of designing algorithms that inexactly compute the inertial proximal gradient point with implementable conditions that still preserve the convergence guarantees of Theorem \ref{thm:convergence}. More precisely, we propose two new inexact inertial--type algorithms, where the second one also features a linesearch procedure along a descent direction of a suitable merit function. The convergence analysis of both algorithms can be performed in the abstract framework provided by Conditions \ref{definition:abstract}. We stress that, due to the inexactness in the computation of the inertial proximal gradient point, our algorithms cannot be cast in the abstract frameworks proposed in previous works.

\subsection{Inexactness criterion of the inertial proximal gradient point}\label{sec:inexact_inertial}

We start our presentation by defining the inexactness criterion for the inertial proximal gradient point, which generalizes the one proposed in \cite{Bonettini-Loris-Porta-Prato-2016,Bonettini-Loris-Porta-Prato-Rebegoldi-2017,Bonettini-Prato-Rebegoldi-2020}.

Given two positive parameters $\alpha_k,\beta_k$, consider the function $h\k:\R^n\times\R^n\times\R^n\to \R\cup\{+\infty\}$ defined as follows:
\begin{equation}\label{hsigma}
h\k(y;x,s) = f_1(y) - f_1(x) + \langle \nabla f_0(x) - \frac{\beta_k}{\alpha_k}(x-s),y-x\rangle + \frac{1}{2\alpha_k}\|y-x\|^2.
\end{equation}
Clearly, the inertial proximal gradient point \eqref{eq:inertial_FB_iterate} is the minimizer of the above function with respect to the first argument, with $x = \xk$, $s = \xkm$.\\
Given $(\xk,s\k)$, we denote by $\hat y\k$ the (exact) minimizer of the function in \eqref{hsigma}
\begin{equation}\label{hydef}
\hat y\k = \argmin_{y\in\R^n}h\k(y;x\k,s\k) = \prox_{\alpha_k f_1}(x\k-\alpha_k\nabla f_0(x\k) +\beta_k (x\k-s\k)).
\end{equation}
\def\hy{ \hat y}
\def\ty{ \tilde y}
\noindent The point $\hat y\k$ is the unique point satisfying the optimality condition
\begin{gather}\label{eq:optimal}
 0\in\partial h\k(\hy\k;\xk,s\k)  \nonumber\\*
 \Updownarrow  \\
 -\frac 1 \alpha_k (\hy\k - x\k+\alpha_k \nabla f_0(x\k)-\beta_k(x\k-s\k)) \in \partial f_1(\hy\k). \nonumber
\end{gather}
Borrowing the ideas in \cite{Bonettini-Loris-Porta-Prato-2016, Bonettini-Prato-Rebegoldi-2021}, we define an approximation of $\hy\k$ as any point $\ty\k\in \dom(f_1)$ such that
\begin{equation}\label{inexact_crit1}
h\k(\ty\k;x\k,s\k)-h\k(\hy\k;x\k,s\k)\leq -\frac \tau 2 h\k(\ty\k;x\k,s\k),
\end{equation} 
\silviacorr{for a given constant $\tau \geq 0$ not depending on $k$. Since $\hyk$ is the unique minimizer of $h\k(\,\cdot\, ;x\k,s\k)$, for $\tau=0$ the previous inequality implies $\ty\k = \hy\k$, therefore the inexactness occurs for $\tau > 0$.}
The above condition is equivalent to the following one:
\begin{equation}\label{inexact_crit2}
h\k(\ty\k;x\k,s\k)\leq\left(\frac{2}{2+\tau}\right)h\k(\hy\k;x\k,s\k)\leq 0,
\end{equation} 
where the rightmost inequality is a consequence of the fact that $\hy\k$ is a minimizer of $h\k (\cdot;x\k,s\k)$ and $h\k (x\k;x\k,s\k)=0$. Therefore, we have $h\k(\ty\k;x\k,s\k)\leq 0$ and condition \eqref{inexact_crit1} can be rewritten in equivalent way as
\begin{equation}\label{inexact_crit3}
0\in\partial_{\epsilon_k}h\k(\ty\k;x\k,s\k), \ \ \mbox{with } \epsilon_k = -\frac\tau 2h(\ty\k;x\k,s\k),
\end{equation}
where $\partial_{\epsilon}h(\ty;x,s)=\{w\in\mathbb{R}^n: h(y;x,s)\geq h(\ty;x,s)+\langle w,y-\ty\rangle-\epsilon, \ \forall y\in\mathbb{R}^n\}$ is the $\epsilon$-subdifferential in the first argument of the convex function $h(\cdot;x,s)$ at point $\ty$ \cite[p. 82]{Zalinescu-2002}. Therefore, the point $\ty\k$ is defined upon a relaxation of the optimality condition \eqref{eq:optimal}, where the subdifferential of $h\k(\cdot;x\k,s\k)$ is replaced by the $\epsilon_k$-subdifferential and the accuracy parameter $\epsilon_k$ is chosen in a specific way, which is crucial for preserving the theoretical convergence properties.

Even if the inclusion \eqref{inexact_crit3} is implicit, a point $\tilde{y}^{(k)}\in \dom(f_1)$ satisfying \eqref{inexact_crit3} can be actually computed in practice { in some special cases} with a well defined, explicit primal--dual procedure, as explained in the following section. 

\silviacorr{\subsubsection{Computation of the inexact inertial proximal gradient point}\label{sec:procedure}
In this section we describe a procedure for computing a point $\ty\k$ satisfying criterion \eqref{inexact_crit3}. Our discussion is mainly based on the approach presented in \cite{Bonettini-Rebegoldi-Ruggiero-2018}, which is a development of \cite{Bonettini-Loris-Porta-Prato-Rebegoldi-2017,Bonettini-Loris-Porta-Prato-2016,Salzo-Villa-2012,Villa-etal-2013}. 
}

\silviacorr{
We consider the case when the convex term in \eqref{composite_problem} has the structure 
\begin{equation}\label{f1form}
f_1(x)=\sum_{i=1}^{p}{ g_i}(M_ix)+\xi(x),
\end{equation}
where, for all $i\in\{1,\ldots,p\}$, $M_i\in\R^{m_i\times n}$ and ${ g_i}:\R^{m_i}\rightarrow \bR$, $\xi:\R^n\rightarrow \bR$ are proper convex functions that are continuous on their domain. Moreover, we assume that $\dom({ g_i}\circ M_i)\supseteq \dom(\xi)$, i.e., $\dom(f_1)=\dom(\xi)$, and $ g_i$, $\xi$ admit a closed form formula for computing their associated proximity operators.} 


\silviacorr{
With the aim of describing a procedure for computing a point satisfying \eqref{inexact_crit3}, we rewrite the convex minimization problem in \eqref{eq:mini_sub_prob}, omitting the iteration index for simplicity, as
$$
\min_{z\in\R^n} \ h(z;x,s)\equiv  \sum_{i=1}^{p}{ g_i}(M_iz)+\xi(z)+\frac{1}{2\alpha}\|z-\bar{x}\|^2+c,
$$ 
where $\bar{x}= x-\alpha \nabla f_0(x)+\beta(x-s)$ and $c=-\frac{\alpha}{2}\|\nabla f_0(x)-\beta/\alpha(x-s)\|^2-f_1(x)$ depend on $x$ and $s$ but not on the optimization variable $z$. Recalling the relation ${ g_i}(M_iz)=\max_{w_i\in\R^{m_i}}\langle w_i,M_iz\rangle-{ g_i}^*(w_i)$, where ${ g_i}^*:\R^{m_i}\rightarrow \bR$ is the convex conjugate of $ g_i$, and plugging it in the previous equation, results in the following primal--dual formulation 
$$
\min_{z\in\R^n}\max_{w\in\R^{m}} \langle w,Mz\rangle-\sum_{i=1}^{p}{ g_i}^*(w_i)+\xi(z)+\frac{1}{2\alpha}\|z-\bar{x}\|^2+c,
$$
where $M=(M_1^T \ M_2^T \ \cdots M_p^T)^T\in\R^{m\times n}$, $m=m_1+\ldots+m_p$, and $w = (w_1^T, w_2^T, \cdots, w_p^T)^T\in\R^m$. Some further simple manipulations lead to the following equivalent reformulation
\begin{align*}
\max_{w\in\R^{m}}\min_{z\in\R^n} \ \xi(z)&+\frac{1}{2\alpha}\|z-(\bar{x}-\alpha M^Tw)\|^2-\frac{1}{2\alpha}\|\bar{x}-\alpha M^Tw\|^2+\\
&+\frac{1}{2\alpha}\|\bar{x}\|^2-\sum_{i=1}^{p}{ g_i}^*(w_i)+c.
\end{align*}
Minimizing with respect to $z$ and using the definition of proximity operator yields the dual problem
\begin{equation}\label{eq:dual}
\max_{w\in\R^{m}} \psi(w;x,s)\equiv G(w;x,s)-\sum_{i=1}^{p}{ g_i}^*(w_i),
\end{equation}
where 
\begin{align*}
G(w;x,s)&=\xi(\prox_{\alpha \xi}(\bar{x}-\alpha M^Tw))\\
&+\frac{1}{2\alpha}\|\prox_{\alpha \xi}(\bar{x}-\alpha M^Tw)-(\bar{x}-\alpha M^Tw)\|^2\\
&-\frac 1{2\alpha} \|\bar x -\alpha M^Tw\|^2+\frac{1}{2\alpha}\|\bar{x}\|^2+c.
\end{align*}
The function $G(\,\cdot\,;x,s)$ is concave, continuously differentiable with $\nabla G(w;x,s)=M \prox_{\alpha \xi}(\bar{x}-\alpha M^Tw)$ \cite[Proposition 12.30]{Bauschke-Combettes-2011}, and its domain is the whole dual space $\R^m$. Note also that, by definition,
\begin{equation*}
h(z;x,s)\geq \psi(w;x,s) \ \ \forall z\in\R^n,w\in\R^m
\end{equation*} 
and the equality holds if and only if $z=\hy$ and $w = \hat w$, where $\hat w$ denotes a solution of the dual problem \eqref{eq:dual}. In particular, the following relations hold:
\begin{eqnarray*}
\hy&=& \prox_{\alpha \xi}(\bar{x}-\alpha M^T\hat w)\\
 \psi(\hat w;x,s)&= &h(\hy;x,s)\\
\psi(w;x,s)&\leq&h(\hy;x,s)\ \ \forall w\in\R^m.
\end{eqnarray*}
Then, if a primal-dual pair $(\ty,\tilde{w})\in\R^n\times\R^m$ satisfies the inequality
\begin{equation}\label{eq:ty_practical}
h(\ty;x,s)\leq \frac{2}{2+\tau} \psi(\tilde{w};x,s),
\end{equation}
the point $\ty$ complies with \eqref{inexact_crit3}. A pair $(\ty,\tilde{w})$ satisfying \eqref{eq:ty_practical} can be computed by proceeding as follows: 
\begin{enumerate}
    \item apply an iterative method to the dual problem \eqref{eq:dual} generating a sequence $\{w^{(\ell)}\}_{\ell\in\N}\subset \dom(\psi)$ such that $\{w^{(\ell)}\}_{\ell\in\N}$ converges to $\hat{w}$ and $\{\psi(w^{(\ell)};x,s)\}_{\ell\in\N}$ converges to $\psi(\hat w;x,s)$; 
    \item define the corresponding primal sequence $\tilde{y}^{(\ell)}=\prox_{\alpha \xi}(\bar{x}-\alpha M^Tw^{(\ell)})$; 
    \item stop the dual iterations when $h(\ty^{(\ell)};x,s)\leq \eta \psi(\tilde{w}^{(\ell)};x,s)$, where $\eta = 2/(2+\tau)$; 
    \item set $\ty = \ty^{(\ell)}$.
\end{enumerate}}

\silviacorr{
Note that the described procedure is well defined: indeed, by continuity, $\ty^{(\ell)}$ converges to $\hy$, and since ${ g_i},\xi$ are assumed to be continuous, $h(\ty^{(\ell)};x,s)$ converges to $h(\hy;x,s)$. If $\dom({ g_i}\circ M_i)\supseteq \dom(\xi)$, as we assumed at the beginning of this section, then all the points $\ty^{(\ell)}$ generated by this procedure (including its output $\ty$) belong to the domain of $f_1$, which coincides with the domain of $h(\cdot;x,s)$. Feasibility is required during the dual iterations for checking the stopping criterion, as well as for the inexactness criterion \eqref{inexact_crit2} to be fulfilled.}

\silviacorr{We remark that the objective function $-\psi(w;x,z)$ of the dual problem \eqref{eq:dual} is the sum of the smooth convex term $-G(w;x,z)$ and the convex function $\sum_{i=1}^p{ g_i}^*(w_i)$, whose proximity operator is easy to compute, as it can be computed by separately evaluating the proximity operators associated to ${ g_i}^*$ ($i=1,...,p$), for which Moreau's identity and, consequently, the proximity operator of $ g_i$, can be exploited. This means that, for example, any (exact) forward--backward method can be applied to the dual problem \eqref{eq:dual} for generating the sequence  $\{w^{(\ell)}\}_{\ell\in\N}$.
}

\subsection{\iPiano{}: inertial inexact proximal algorithm for nonconvex optimization}\label{sec:iPiano}

In this section we propose a generalization of the inertial method iPiano, first proposed in \cite{Ochs-etal-2014} and further developed in \cite{Ochs-2019}, introducing the possibility of an inexact computation of the inertial proximal gradient point.

Our proposed method is reported in Algorithm \ref{algo:iPiano_inexact_Silvia} and denoted as i$^2$Piano (inertial inexact Proximal algorithm for nonconvex optimization). Let us describe the i$^2$Piano iteration in detail. {\sc STEP 1--4} determine the stepsize $\alpha_k$ and inertial parameter $\beta_k$ at iteration $k$. Given the parameters $\alpha_k,\beta_k$, {\sc STEP 5} seeks to find a possibly inexact inertial proximal point, i.e., an inexact minimizer of the function \eqref{hsigma}. 
According to \eqref{inexact_crit3}, the i$^2$Piano iterate is any point $\xkk=\tilde{y}^{(k)}$ such that
\begin{equation*}\label{defxkkiPiano}
0\in \partial_{\epsilon_k}\hk(\xkk;\xk,\xkm),\ \ \ \mbox{ with }\epsilon_k = -\frac\tau 2 \hk(\xkk;\xk,\xkm)
\end{equation*}
for some fixed constant $\tau\geq 0$. When $\tau = 0$, we recover the exact inertial proximal gradient point provided by the iPiano method \cite{Ochs-2019,Ochs-etal-2014}. If $\tau>0$, as explained in Section \ref{sec:procedure}, \textsc{Step 5} of \iPiano{} can be practically implemented with an inner loop consisting of an iterative optimization method applied to the dual of problem $\min_{y\in\R^n} h\k(y;\xk,\xkm)$, until the stopping condition \eqref{eq:ty_practical} is met. In the implementation of \iPiano{}, besides the stepsize $\alpha_k$ and the inertial parameter $\beta_k$, a further parameter, $L_k$, is introduced (cf. {\sc STEP 6}) with the aim of estimating a local Lipschitz constant of $\nabla f_0$ that allows us to take larger steps. In particular, $L_k$ is successively increased by a factor $\eta>1$ until the following descent condition holds
\begin{equation}\label{eq:descent_inexact}
f_0(\xkk)\leq f_0(x^{(k)})+\langle \nabla f_0(x^{(k)}),\xkk-x^{(k)}\rangle+\frac{L_k}{2}\|\xkk-x^{(k)}\|^2.
\end{equation}
\silviacorr{Clearly, if the Lipschitz constant $L$ is known, the previous inequality is satisfied with $L_k\geq L$ for all $k$, and the loop between {\sc STEP 2} and {\sc STEP 6} can be skipped. In this case, also the choice of the inertial parameter $\beta_k$ and the corresponding steplength $\alpha_k$ could be simplified. For example, $\{\beta_k\}_\kinN$ can be set as a constant sequence $\beta_k\equiv \beta < 1$, or as a prefixed sequence such that $\beta_k < 1$ and $\lim_k\beta_k = 1$, whereas $\alpha_k$ can be selected as in {\sc Step 4} with $L_k=L$. 
}

\silviacorr{
The choice of the parameters in Algorithm \iPiano{} is very similar to the one proposed in \cite{Ochs-2019}, and aims at adaptively estimating the Lipschitz constant, which is quite common in forward--backward methods \cite{Bonettini-Loris-Porta-Prato-Rebegoldi-2017,Calatroni-Chambolle-2019,Ochs-etal-2014,Ochs-2019,Scheinberg-2014}. We stress that the main novelty in Algorithm \iPiano{} is not on the parameters choice, but in the possibility of inexactly computing the proximity operator according to an implementable inexactness criterion, which, as far as we know, is new in the framework of inertial methods for nonconvex optimization.}

\begin{algorithm}[h!]
\caption{\iPiano{}: inertial inexact proximal algorithm for nonconvex optimization.}\label{algo:iPiano_inexact_Silvia}
Choose $x^{(-1)}, x^{(0)}\in\dom(f_1)$, $\delta\geq \gamma>0$,  $\eta>1$, $0<L_{min}\leq L_{max}$, $\tau\geq 0$. Set $\theta = 2/(\sqrt{2+\tau}+\sqrt{\tau})^2$ and choose $\omega\in[0,1)$ if $\tau > 0$, $\omega\in[0,1]$ if $\tau = 0$.\\
{\textsc FOR $k=0,1,\ldots$}
\begin{itemize}
\item[] 
\begin{description}
\item[\mdseries\textsc{   Step 1.}] Choose $L_k\in[L_{min},L_{max}]$.
\item[\mdseries\textsc{   Step 2.}] Set $\displaystyle b_k = \frac{L_k+2\delta}{L_k+2\gamma}$.
\item[\mdseries\textsc{   Step 3.}] Set $\beta_k=\displaystyle\frac{1+\theta\omega}{2}\cdot\frac{b_k-1}{b_k-\frac 1 2}$.
\item[\mdseries\textsc{   Step 4.}] Set $\alpha_k=\displaystyle \frac{1+\theta\omega-2\beta_k}{L_k+2\gamma}$.
\item[\mdseries\textsc{   Step 5.}] Compute $\tyk$ such that
\begin{equation}\nonumber
0\in \partial_{\epsilon_k}\hk(\tyk;\xk,\xkm),\ \ \ \mbox{ with }\epsilon_k = -\frac\tau 2 \hk(\tyk;\xk,\xkm).
\end{equation}
\item[\mdseries\textsc{   Step 6.}] Check the local descent:
\begin{description}
\item[\mdseries\textsc{If}] 
$\displaystyle
f_0(\tyk)\leq f_0(x^{(k)})+\langle \nabla f_0(x^{(k)}),\tyk-x^{(k)}\rangle+\frac{L_k}{2}\|\tyk-x^{(k)}\|^2
$
\begin{description}
\item[-] Set $\xkk = \tyk$.
\end{description}
\item[\mdseries\textsc{Else}] 
\begin{description}
\item[]
\item[-]Set $L_k = \eta L_k$.
\item[-]Go to \textsc{Step 2.}
\end{description} 
\end{description}
\end{description}
\end{itemize}
{\textsc END}
\end{algorithm}

\silviacorr{Under some standard boundedness assumptions and within the Kurdyka--Lojasiewicz framework employed within the abstract framework of Section \ref{sec:abstract}, we can prove that the sequence generated by \iPiano{} converges to a stationary point, as stated below. 
\begin{theorem}\label{thm:i2piano}
Let ${\mathcal F}:\R^n\times \R \to \bar{\R}$ be the { merit} function defined as
\begin{equation}\label{eq:surro}
{\mathcal F}(u,\rho) = f(u) +\frac 1 2 \rho^2, \quad \forall u\in\mathbb{R}^n, \ \rho\in\mathbb{R}.
\end{equation}
Suppose that $\mathcal{F}$ is a KL function and assume that the sequence $\{x^{(k)}\}_{k\in\mathbb{N}}$ generated by \iPiano{} is bounded. Then, $\{x^{(k)}\}_{k\in\mathbb{N}}$ converges to a stationary point of $f$.
\end{theorem}
In order not to excessively slowing down the reading of this section, we postpone the proof of Theorem \ref{thm:i2piano} to Section \ref{sec:conviPiano}. 
}

\begin{remark}\label{remark:KL}
We underline that $\mathcal{F}$ is a KL function if, for instance, $f$ and $\frac{1}{2}\|\cdot\|^2$ are definable in the same $o-$minimal structure \cite[Definition 7]{Bolte-etal-2007b}. Indeed functions definable in an $o-$minimal structure satisfy the KL property on their domain \cite[Theorem 11]{Bolte-etal-2007b} and $o-$minimal structures are closed with respect to the sum, see \cite[Remark 5]{Bolte-etal-2007b} and references therein. Examples of functions definable in an $o-$minimal structure are semialgebraic, \silviacorr{real analytic functions and subanalytic functions which are continuous in their closed domain.}
\end{remark}

\begin{remark}\label{remark:boundedness}
Theorem \ref{thm:i2piano} requires the boundedness of the iterates as hypothesis. A standard way to assert such a condition is when the Lyapunov function $\Phi$ defined in \eqref{eq:Phi} is coercive, since this assumption combined with the descent property \eqref{H1iPiano_inexact} guarantees that the sequence $\{(x^{(k)},x^{(k-1)})\}_{k\in\mathbb{N}}$ is included in a (bounded) level set of the coercive function $\Phi$. 
\end{remark}

\subsection{\iPianoLA{}: inertial proximal inexact line--search algorithm}\label{sec:vmilaxs}

In the following we introduce a novel algorithm combining a line--search along the descent direction and an inertial proximal-gradient step as a special case of our abstract scheme. 

A line--search procedure for the objective function $f$ requires a descent direction $d\in\mathbb{R}^n$, i.e., a vector such that the directional derivative $f'(x;d)=\lim_{\lambda\downarrow 0}(f(x+\lambda)-f(x))/\lambda$ is negative. If $\beta_k=0$ and $\ty\k$ satisfies \eqref{inexact_crit1}, then the vector $\ty\k-\xk$ is a descent direction for $f$ at $x\k$, as explained in \cite{Bonettini-Loris-Porta-Prato-2016,Bonettini-Loris-Porta-Prato-Rebegoldi-2017}
\silviacorr{In fact, if $\beta_k=0$, the inequality $h\k(\ty\k;\xk,s\k)\leq 0$ would be enough to guarantee the descent property of the vector $\ty\k-\xk$, even though it would not be sufficient to guarantee the convergence of the iterates, as shown in \cite{Bonettini-Loris-Porta-Prato-2016}.} Unfortunately, none of the two previous conditions can guarantee that $\tyk-\xk$ is a descent direction for $f$ at $x^{(k)}$ when $\beta_k > 0$, i.e., when inertia is incorporated in the iterative scheme.

In this section we show that in the general case $\beta_k	\geq 0$, the point $\ty\k$ can still be used to define a descent direction for a suitable merit function. Then, we propose a line--search procedure along this direction that enable us to define a descent algorithm such that the merit function monotonically decreases along the iterates. Differently from the backtracking procedure in \iPiano{}, the proposed line--search requires to solve the minimization subproblem \eqref{eq:mini_sub_prob} only once per iteration. Finally, we show that the new algorithm can be analyzed in the framework of Section \ref{sec:abstract}, in order to prove the convergence of the iterates to a stationary point of $f$. In this case, unlike \iPiano{}, one of the two merit functions involved in the abstract scheme will play an active role in the algorithm, as it will be explicitly computed at each iteration to determine the new point.

We define the merit function $\Phi$ appearing in \ref{H1}--\ref{H2} as follows:
\begin{equation}\label{Phidef}
\Phi:\R^n\times \R^n \to \R\cup \{+\infty\}, \ \ \  \Phi(x,s) = f(x) + \frac 1 2 \|x-s\|^2,
\end{equation}
where the variable $s$ will be considered as an actual optimization variable, independent on $x$. The function $\Phi$ can be decomposed as $\Phi(x,s) = \Phi_0(x,s) + \Phi_1(x,s)$, where
\begin{equation}\nonumber
 \Phi_0(x,s) = f_0(x) + \frac 1 2 \|x-s\|^2,\ \ \Phi_1(x,s) = f_1(x).
\end{equation}
The function $\Phi_0$ is differentiable with gradient
\begin{equation}\nonumber
\nabla \Phi_0(x,s) = \begin{pmatrix} \nabla_x \Phi_0(x,s) \\ \nabla_s \Phi_0(x,s) \end{pmatrix} = \begin{pmatrix} \nabla f_0(x) + x-s\\ s-x\end{pmatrix} .
\end{equation}
It is easy to see that $\nabla \Phi_0(x,s)$ is Lipschitz continuous and in particular it holds that
\begin{equation}\label{nablaH0lip}
\|\nabla \Phi_0(x,s) - \nabla \Phi_0(\bar x,\bar s)\| \leq { L_{\Phi_0}} \left\| \begin{pmatrix} x-\bar x\\ s-\bar s\end{pmatrix}\right\|, \ \ \forall x,\bar x, s,\bar s,
\end{equation}
with ${ L_{\Phi_0}}=L+2$, where $L$ is the Lipschitz constant of $\nabla f_0$.\\
Given a vector $d\in\R^{n}\times \R^n$
\begin{equation}
d = \begin{pmatrix} d_x\\d_s\end{pmatrix},
\end{equation}
the directional derivative of $\Phi$ at the point $(x,s)$ with respect to the direction $d$ can be written as
\begin{eqnarray*}
\Phi'(x,s;d_x,d_s) &=& \Phi_0'(x,s;d_x,d_s)+\Phi_1'(x,s;d_x,d_s) = f_1'(x;d_x) + \langle \nabla \Phi_0(x,s),d\rangle 	\\
&= & f_1'(x;d_x) + \langle \nabla f_0(x) + x-s, d_x\rangle + \langle s-x,d_s\rangle ,
\end{eqnarray*}
which always exists thanks to the convexity of $f_1$.\\
A vector $d\in\R^{2n}$ is called a descent direction for $\Phi$ at $(x,s)$ when
\begin{equation}\nonumber
\Phi'(x,s;d_x,d_s)< 0.
\end{equation}
Assume now that the vector $d_x$ has the form $d_x = y-x$, where $y$ is a point belonging to the domain of $f_1$; then, from \cite[Theorem 23.1]{Rockafellar-1970} we have
\begin{equation}\label{xs1}
\Phi'(x,s;y-x,d_s) \leq f_1(y)-f_1(x) + \langle \nabla f_0(x) + x-s, d_x\rangle + \langle s-x,d_s\rangle .
\end{equation}
The above inequality holds independently on the form of $d_s$.

\subsubsection{Defining a descent direction for the merit function}\label{sec:descent}
\def\sk{ s^{(k)}}
\def\skk{ s^{(k+1)}}
\def\wk{ w^{(k)}}
\def\yk{y^{(k)}}
\def\tyk{\tilde y^{(k)}}

Assume that $(\xk,\sk)$ is a given point in $\dom(f_1)\times\R^n$, while $\alpha_k$, $\beta_k$ are two given parameters. Let the function $\hk(y;x,s)$ be defined as in \eqref{hsigma}. 
Given a fixed tolerance parameter $\tau \geq 0$, we denote by $\tyk$ any point in $\dom(f_1)$ satisfying \eqref{inexact_crit3}.
%
Given the parameter $\gamma_k\geq 0$, consider
\begin{equation}\label{dk_xs}
\dk = \begin{pmatrix}  d_x\k \\ d_s\k \end{pmatrix},
\end{equation}
where
\begin{eqnarray}
d_x\k &=&\tyk-\xk\label{dx_xs}\\
d_s\k &=& \left(1+\frac{\beta_k}{\alpha_k}\right)(\tyk-\xk) + \gamma_k(\xk-\sk)\label{ds_xs}.
\end{eqnarray}
It can be shown (see Lemma \ref{lemma:1_xs} in Section \ref{subsec:convergence_ipianola}) that $\dk$ is a descent direction for $\Phi$ at $(\xk,\sk)$ and, therefore, we can introduce a backtracking procedure along it to seek for a sufficient decrease of the merit function.
In particular, given a point $(\xk,\sk)$ and the direction $d\k$ \eqref{dx_xs}--\eqref{ds_xs}, our proposed line--search algorithm computes a positive parameter $\lamk^+$ satisfying the following generalized Armijo inequality.
\begin{equation}\label{ine_arm}
\Phi(\xk+\lambda^+_kd_x\k,\sk+\lambda^+_kd_s\k)\leq \Phi(\xk,\sk) + \sigma\lambda^+_k\Delta_k, \quad \sigma\in(0,1),
\end{equation}
where 
\begin{equation*}
    \Delta_k = h\k(\tyk;\xk,\sk)-\gamma_k\|\xk-\sk\|^2.
\end{equation*}
The implementation of this rule via a backtracking procedure is given in Algorithm \ref{algo:LS}.
\begin{algorithm}[t!]
\caption{Armijo line--search}\label{algo:LS}
{\textsc INPUT}: $(\xk,\sk)\in\R^{n}\times\R^n$, $\dk$, $\Delta_k$ as in \eqref{defDeltak}, $\sigma,\delta\in (0,1)$
\begin{itemize}
\item[] Set $\lambda^+ = 1$.
\item[] {\textsc WHILE $\Phi(\xk+\lambda^+d_x\k,\sk+\lambda^+d_s\k)> \Phi(\xk,\sk) + \sigma\lambda^+\Delta_k$}  
\begin{itemize}
\item[] Set $\lambda^+ =\delta\lambda^+$
\end{itemize}
\item[] Set $\lamk^+ =\lambda^+$
\end{itemize}
{\textsc END}\\
{\textsc OUTPUT:} $\lamk^+$.
\end{algorithm}
In Section \ref{subsec:convergence_ipianola} we show that Algorithm \ref{algo:LS} is well posed, i.e., it terminates in a finite number of steps. 

\begin{algorithm}[t!]
\caption{\iPianoLA{}: inertial proximal inexact line--search algorithm}\label{algo:VMILA_xs}
{\textsc INPUT}: $(x^{(0)},s^{(0)})\in\dom(f_1)\times\R^n$, $\sigma\in (0,1)$, $0<\alpha_{min}\leq \alpha_{max}$, $\beta_{max}>0$, $0<\gamma_{min}\leq \gamma_{max}$, $\tau \geq 0$.\\
{\textsc FOR $k=0,1,\ldots$}
\begin{itemize}
\item[]\begin{description}[labelwidth=4em,leftmargin =\dimexpr\labelwidth+\labelsep\relax, font=\mdseries]
\item[{\textsc STEP 1.}] Choose $\alpha_k\in[\alpha_{min},\alpha_{max}]$, $\beta_k\in [0,\beta_{max}]$
\item[{\textsc STEP 2.}] Compute $\tyk$ such that 
\begin{equation*}
0\in \partial_{\epsilon_k}\hk(\tyk;\xk,\sk),\ \ \ \mbox{ with }\epsilon_k = -\frac\tau 2 \hk(\tyk;\xk,\sk).
\end{equation*}
\item[{\textsc STEP 3.}] Choose $\gamma_k\in[\gamma_{min},\gamma_{max}]$.
\item[{\textsc STEP 4.}] Compute $\Delta_k= \hk(\tyk;\xk,\sk) - \gamma_k \|\xk-\sk\|^2$.
\item[{\textsc STEP 5.}] 
Compute the search direction
\begin{eqnarray*}
d_x\k &=& \tyk-\xk\\
d_s\k &=&\left( 1 + \frac{\beta_k}{\alpha_k}\right)(\tyk-\xk) + \gamma_k(\xk-\sk)\\
\end{eqnarray*}
\item[{\textsc STEP 6.}] Compute $\lamk\in (0,1]$ such that
\begin{equation}\nonumber
\Phi(\xk+\lambda_kd_x\k,\sk+\lambda_kd_s\k)\leq \Phi(\xk,\sk) + \sigma\lambda_k\Delta_k
\end{equation}
with the line--search backtracking algorithm.
\item[{\textsc STEP 7.}] Define the new point as
\begin{equation}\nonumber
(\xkk,\skk) = \left\{\begin{array}{l}
(\tyk,\xk) \ \ \mbox{ if } \Phi(\tyk,\xk)\leq \Phi(\xk,\sk) + \sigma\lambda_k\Delta_k\\
(\xk+\lambda_kd_x\k,\sk+\lambda_kd_s\k) \ \ \mbox{ otherwise}
\end{array}
\right.
\end{equation}
\end{description}
\end{itemize}
{\textsc END}\\
\end{algorithm}

The descent direction and the backtracking procedure described above are at the basis of the new algorithm, named \iPianoLA{} (inertial Proximal inexact line--search algorithm), which formally consists in a descent method for the merit function $\Phi(x,s)$ and exploits the inertial inexact proximal gradient point for defining the search direction. In particular, it generates a sequence of iterates $\{(\xk,\sk)\}_\kinN$ and a sequence of steplength parameters $\{\lamk\}_{\kinN}$ fulfilling the following decrease condition 
\begin{equation*}
\Phi(\xkk,\skk)\leq \Phi(\xk,\sk)+\sigma\lamk\Delta_k\ \ \ \mbox{ and } \ \ \
\Phi(\xkk,\skk)\leq \Phi(\tyk,x^{(k)}).
\end{equation*} 
The connection with the inertial methods is in fact that, when $(\xkk ,\skk)=(\tyk,\xk)$ is selected at {\textsc{STEP 7}}, the following iteration will consist of an actual inertial step.
In practice, the condition at {\textsc{STEP 7}} can be considered as an alternative acceptance rule for the inexact inertial proximal gradient point, having a similar role than the condition at {\textsc{STEP 6}} of \iPiano{} (see also \eqref{eq:descent_inexact}). The main difference is that here the acceptance condition is based on the Armijo inequality, while the one in \iPiano{} is based on the Descent Lemma.

Notice that the inexact evaluation of the proximity operator in \iPianoLA{} is required only once per iteration, unlike in \iPiano{}, where it is needed at each step of the loop for selecting the parameter $L_k$, until inequality \eqref{eq:descent_inexact} is satisfied. The Armijo condition also results in a larger freedom of choosing the parameters $\alpha_k,\beta_k$, which here satisfy very minimal assumptions. A possible strategy to choose these parameters preserving both the theoretical prescriptions and the benefits deriving from the presence of an inertial step is described in Section \ref{sec:numerical}.

\silviacorr{The Armijo line--search strategy is very well established in optimization, however, to the best of our knowledge, its use in combination with an inertial/heavy-ball step is completely new. Moreover, even if the use of a { merit} function is quite common in the theoretical analysis of optimization methods, here we propose to explicitly compute it in the algorithm implementation, enforcing its decrease by means of the line--search procedure.}

\silviacorr{Under the same assumptions stated for \iPiano{}, we can show the convergence of \iPianoLA{} to a stationary point (the proof is postponed in Section \ref{subsec:convergence_ipianola}).}

\begin{theorem}\label{thm:convipianola}
Suppose that the function $\mathcal{F}$ defined in \eqref{eq:surro} is a KL function. Moreover, assume that the sequence $\{x^{(k)}\}_{k\in\mathbb{N}}$ generated by \iPianoLA{} is bounded. Then, $\{x^{(k)}\}_{k\in\mathbb{N}}$ converges to a stationary point of $f$.
\end{theorem}

We refer the reader to Remark \ref{remark:KL}-\ref{remark:boundedness} for conditions on $f$ guaranteeing that $\mathcal F$ is a KL function and the sequence $\{x^{(k)}\}_{k\in\mathbb{N}}$ is bounded.


\section{Numerical illustration}\label{sec:numerical}
\silviacorr{The aim of this section is to apply our proposed methods to a couple of difficult problems arising in image restoration, by implementing them in accordance with the theoretical guarantees stated in Section \ref{sec:convergence}. }

The goal of image restoration is to recover a good quality image from a noisy blurred one. Following the variational approach, the clean image is obtained by solving an optimization problem with the structure \eqref{composite_problem}, where the objective function includes a measure of the data fidelity and a regularization/penalization term, incorporating all the a priori information on the desired solution. In addition, a nonnegativity constraint is often imposed for physical reasons. Then, the variational model has the following form: 
\begin{equation}\label{eq:model}
\min_{x\in\mathbb{R}^n_{\geq 0}} {\mathcal D}(Hx,g) +{\mathcal R}(x),
\end{equation}
where $g\in\R^n$ is the noisy blurred data, $H\in\R^{n\times n}$ represents the blurring operator, $\mathcal R:\R^n\to\R$ is the regularization term, and ${\mathcal D}(\cdot,\cdot) $ is the data discrepancy function. Moreover, $\mathbb{R}^n_{\geq 0}=\{x\in\mathbb{R}^n: x_i\geq 0, i=1,\ldots,n\}$ is the nonnegative orthant. 
In the following, we will consider two instances of the image restoration model \eqref{eq:model}, simulating two different kinds of noise, impulse noise and signal dependent Gaussian noise.

\subsection{Image denoising and deblurring in presence of impulse noise}\label{sec:first_test}

When data suffers of impulse noise, the more suitable function to measure the data discrepancy is the $\ell_1-$norm: 
\begin{equation}\nonumber
{\mathcal D}(Hx,g) = \|Hx-g\|_1.
\end{equation}
On the other side, with the aim to preserve some sharpness in the restored image, an edge preserving regularization term has to be included in the variational model. 
In this section we consider as regularization function the one proposed in \cite{Chen-Pock-Ranftl-Bischof-2013,Chen-Ranftl-Pock-2014b}, i.e.,
\begin{equation}\nonumber
{\mathcal R}(x) = \rho\sum_{\ell=1}^q{ \theta_\ell}\sum_{i=1}^n\log(1+(K_\ell x)_i^2),
\end{equation}
where the matrices $K_\ell\in\R^{n\times n}$ correspond to a convolution with a given filter $k_\ell$, while $\rho,\theta_\ell$ are positive parameters. In particular, the set of $48$ filters $k_\ell$ of size $7\times 7$ and corresponding coefficients $\theta_\ell$ have been computed with the approach proposed in \cite{Chen-Ranftl-Pock-2014b}. Finally,  it is possible to prove that ${\mathcal R}(x)$ has Lipschitz-continuous gradient, using the same arguments as in \cite{Bonettini-Prato-Rebegoldi-2021}. Then, setting $f_0(x) = {\mathcal R}(x)$ complies with assumptions \ref{assii}-\ref{assiii}. The nonnegativity constraint, expressed by means of the indicator function of the nonnegative orthant $\iota_{\geq 0}(x)$, can be included in the convex, nonsmooth term of the objective function, i.e., $f_1(x) = \|Hx-g\|_1 + \iota_{\geq 0}(x)$. The regularization parameter $\rho$ has been manually tuned in order to have a good quality restoration. Its value has been set equal to 0.08 for all the runs. Note that $f_1$ does not have a closed-form proximal operator, therefore it is necessary to employ implementable inexactness criteria as the one proposed in Section \ref{sec:inexact_inertial}. \silviacorr{Due to the nonconvexity and the lack of a closed form formula for the proximity operator, this problem is really challenging.}

The proposed algorithms have been implemented in { Matlab R2019a} on a laptop equipped with a { 2.60 GHz Intel Core i7-4510U} processor and { 8 GB} of RAM; the Matlab code is available online at \cite{iPila2021}. The parameters of Algorithm \iPiano{} have been set as $\delta=0.5$, $\gamma=10^{-5}$, $\eta=1.5$, $\omega=0.95$. The estimate of the Lipschitz constant $L_k$ is updated in a nondecreasing way. In particular, the initial value $L_0$ is set as an input parameter; then, at {\textsc{Step 1}} of each iteration, the first tentative value is set as $L_k=L_{k-1}$. This value is possibly increased until inequality \eqref{eq:descent_inexact} is met. Actually, more sophisticated updating rules for this parameter could be adopted; however the objective function of the considered image restoration problem is very costly to evaluate, therefore a more conservative parameters selection rule has shown to be more convenient. As for Algorithm \iPianoLA{}, the parameters settings aim to mimic that of the inertial method \iPiano{}. Indeed, introducing the additional parameters $\delta,\gamma,L_k,b_k>0$, with $b_k= \frac{L_k+2\delta}{L_k+2\gamma}$, we set $\alpha_k,\beta_k,\gamma_k$ as follows 
\begin{equation*}
\beta_k= \frac{b_k-1}{b_k-\frac 1 2}, \qquad \qquad \alpha_k=2 \frac{1-\beta_k}{L_k+2\gamma}, \qquad  \qquad  \gamma_k= \gamma.
\end{equation*}
This choice is motivated by the following arguments. If $L_k$ is a good local approximation of the Lipschitz constant satisfying condition 
\begin{equation}\label{eq:Lk_numerical}
f_0(\tyk)\leq f_0(x^{(k)})+\langle \nabla f_0(x^{(k)}),\tyk-x^{(k)}\rangle+\frac{L_k}{2}\|\tyk-x^{(k)}\|^2,
\end{equation}
then reasoning as in the proof of Proposition \ref{prop:fund1}, and choosing $\delta = 0.5$ and $\gamma_k=\gamma$, we obtain
$$
\Phi(\tyk,\xk) \leq \Phi(\xk,\sk) + \Delta_k + \frac 1{2\alpha_k}\|\tyk-\xk\|^2.
$$
Hence, if $L_k$ satisfies \eqref{eq:Lk_numerical}, the point $(\tyk,\xk)$ will be likely accepted at {\textsc{Step 7}}, as also confirmed by the numerical experience. This reasoning suggests to implement algorithm \iPianoLA{} as follows. We check the condition $\Phi(\tyk,\xk)\leq \Phi(\xk,\sk) + \sigma \Delta_k$ right after {\textsc{Step 4}}: if the condition holds, then the steps from 4 to 7 are skipped in order to avoid unnecessary computations, and the next point is directly defined as $(\xkk,\skk) = (\tyk,\xk)$; otherwise, the value $L_k$ is increased by a factor $\eta=1.5$, and the line--search in steps 5--7 is performed in order to compute the next point. By possibly increasing $L_k$, we aim at improving the chances that $(\tilde{y}^{(k+1)},x^{(k+1)})$ is accepted at the next iteration, thus reducing the computational time due to the line--search reductions steps. \silviacorr{We point out that this procedure for computing the parameters in \iPianoLA{} complies with all the theoretical prescriptions in Section \ref{sec:convergence}.} For \iPianoLA{}, the parameter in the Armijo condition, is set to $\sigma = 10^{-4}$, while $\gamma_k=\gamma=10^{-5}$ for all $k$. The initial estimate $L_0$ of the Lipschitz constant has been set equal to one for both algorithms.

For both \iPiano{} and \iPianoLA{}, the inexact proximal point $\tilde{y}^{(k)}$ is computed by approximately solving the dual of problem $\min_{y\in\R^n} h\k(y;\xk,\sk)$, which is a quadratic problem with simple constraints, with FISTA (more details can be found in \cite{Bonettini-Prato-Rebegoldi-2021} and references therein). The accuracy of the approximation is controlled by the parameter $\tau$: in our experiments we set $\tau = 10^{6}$, which corresponds to a good balancing of the computational complexity among inner and outer iterations. An extensive performance assessment of the algorithms with respect to this and other parameters is out of the scope of this paper, and it will be subject of future research.

The deblurring test problem has been obtained by first artificially blurring a good quality image, then simulating impulse noise on the 15\% of the pixels with \verb"imnoise". The clean and the noisy image are reported in Figure \ref{fig:3} (a) and (b). Assuming reflective boundary conditions, matrix-vector multiplications involving $H$ and $H^T$ can be implemented efficiently with the DCT transform. In particular, each inner (dual) iteration requires the computation of two matrix-vector products of this kind. In fact, in our experiments, only one or two inner iterations per outer iteration are, in general, needed to satisfy the inner stopping criterion.

We compare our proposed algorithms to an inexact version of the standard forward--backward method with backtracking in \cite{Beck-Teboulle-2009b} (denoted as iISTA in the following), and with the variable metric line--search based method denominated VMILAn \cite{Bonettini-Loris-Porta-Prato-Rebegoldi-2017,Bonettini-Prato-Rebegoldi-2021} equipped with its standard parameters settings. The inexact computation of the proximal gradient point in iISTA and VMILAn is implemented exactly in the same way as for \iPiano{} and \iPianoLA{}. Moreover, the initial estimate of the Lipschitz constant in iISTA is set to one, i.e., it is the same choice made for \iPiano{} and \iPianoLA{}.

For comparing different algorithms on a nonconvex problem, we adopt the same approach employed in \cite{Chouzenoux-etal-2014,Ochs-2019}: we first numerically estimated the ``optimal'' value $f^*$ as the smallest function value among the ones obtained by running each algorithm for a huge number of iterations. Then, we run again the three algorithms, computing at each iteration the relative difference of the objective function value with respect the reference value $f^*$. The result of this comparison is depicted in Figure \ref{fig:4}.
\begin{figure}
\begin{center}
\begin{tabular}{ccc}
\includegraphics[scale=0.3]{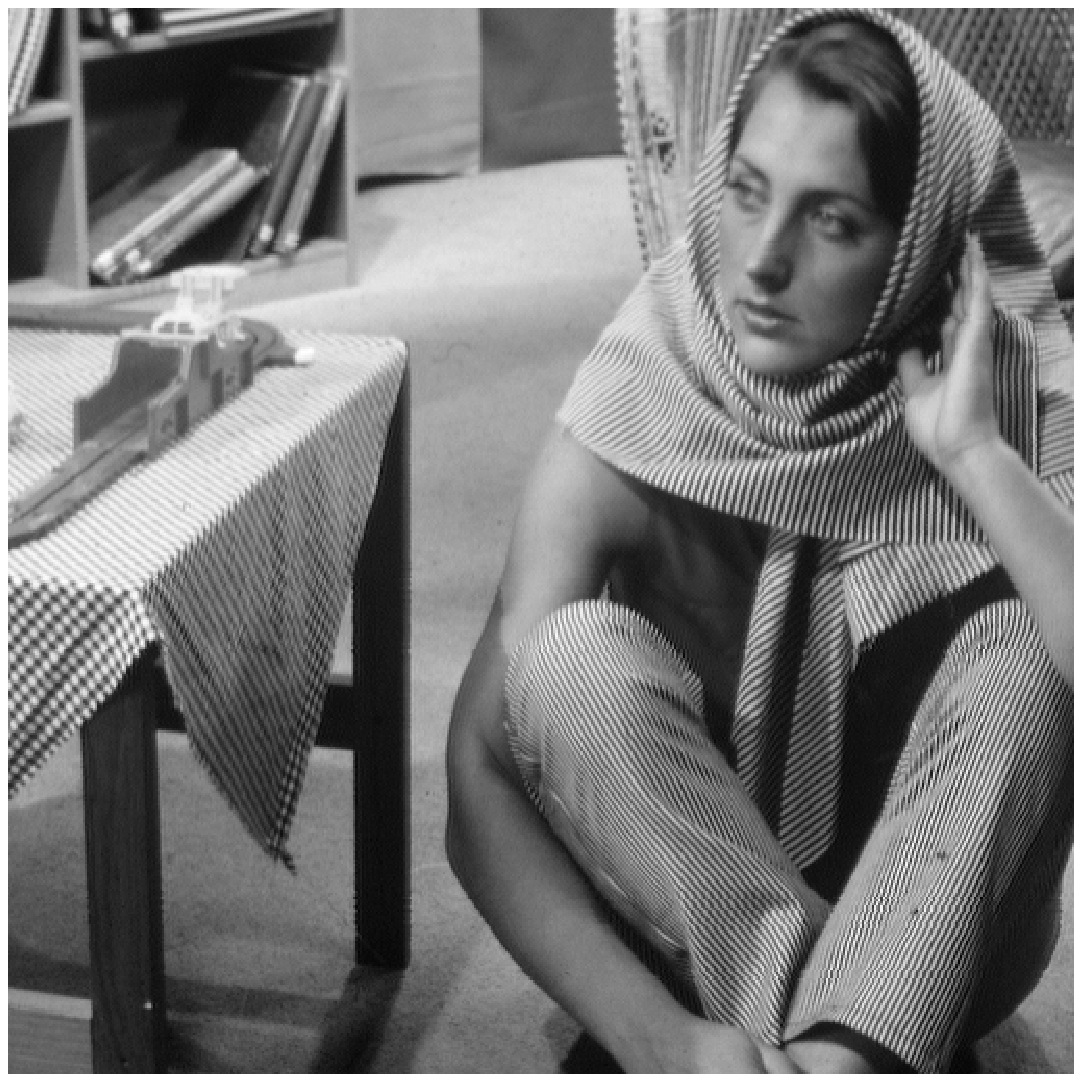}&\includegraphics[scale=0.3]{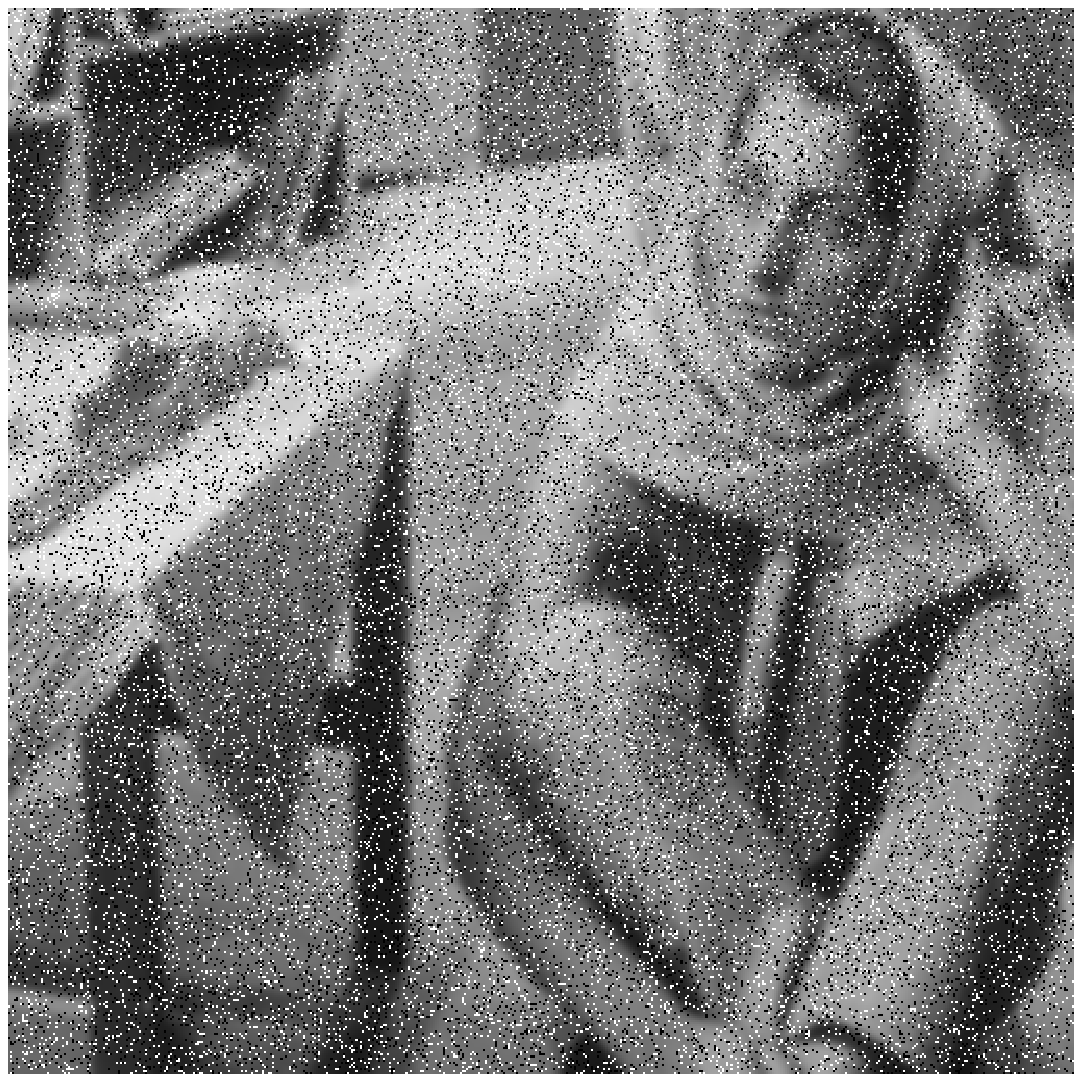}
&\includegraphics[scale=0.3]{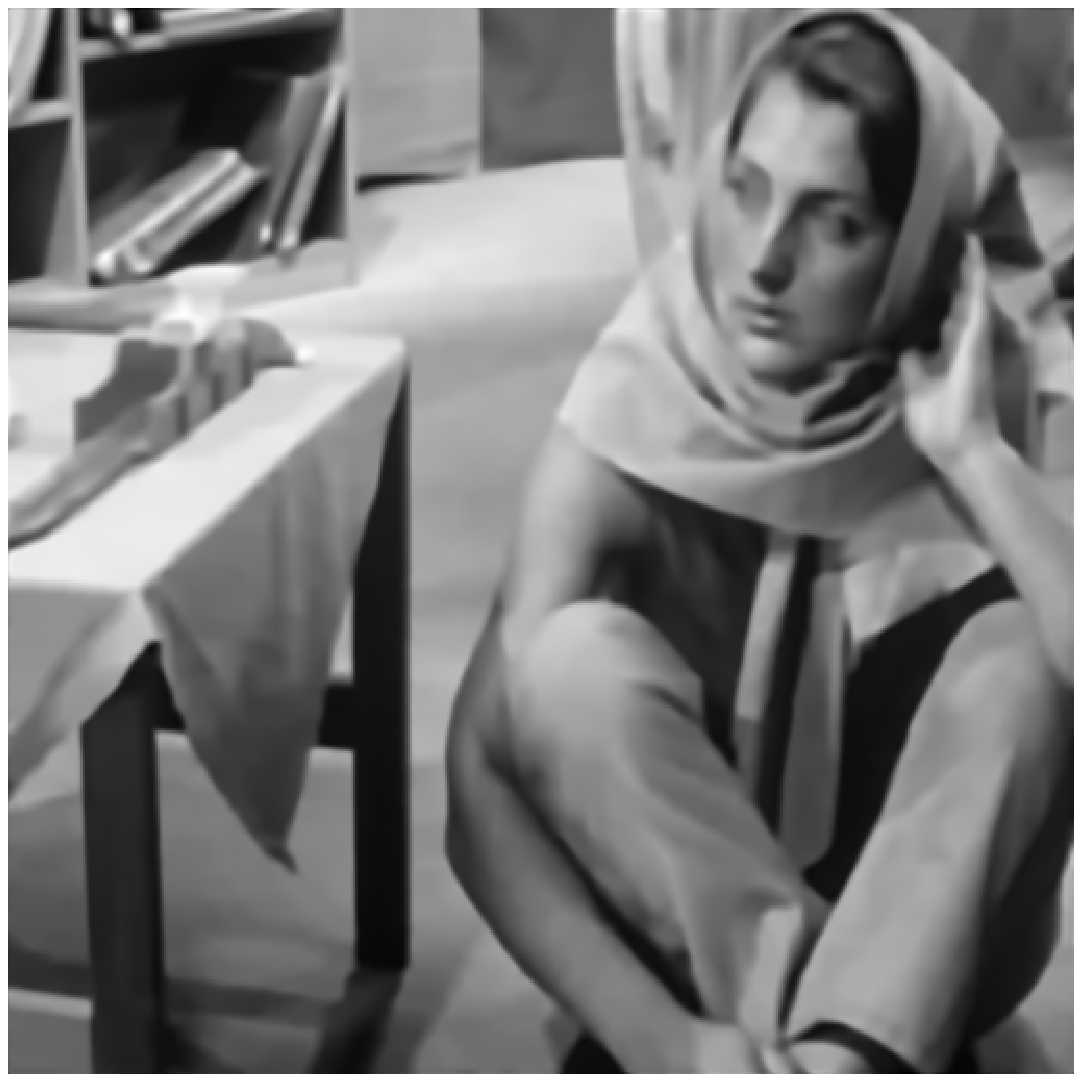}\\
(a) & (b) & (c)\\
\includegraphics[scale=0.5]{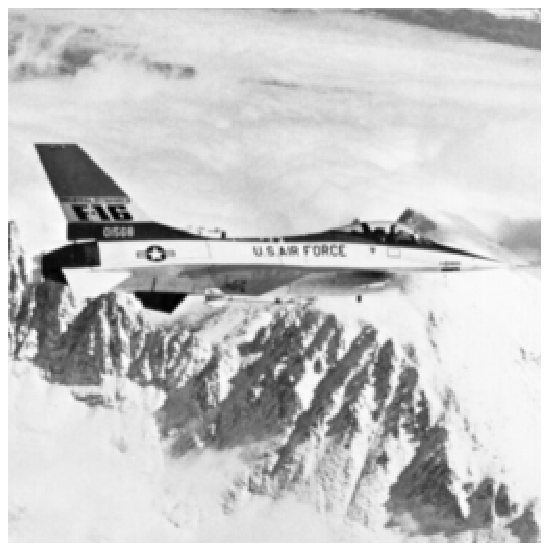}&\includegraphics[scale=0.5]{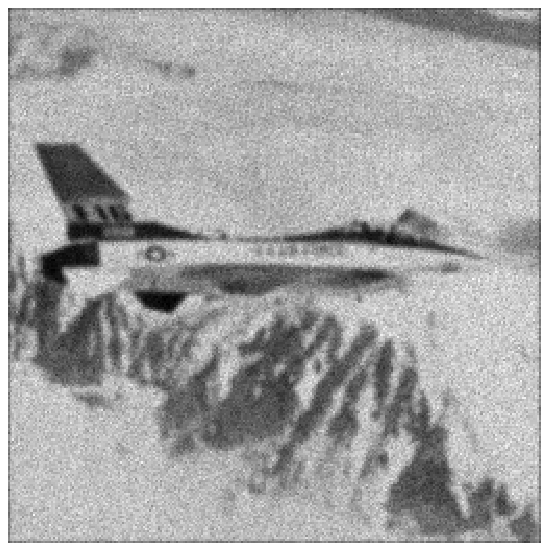}
&\includegraphics[scale=0.5]{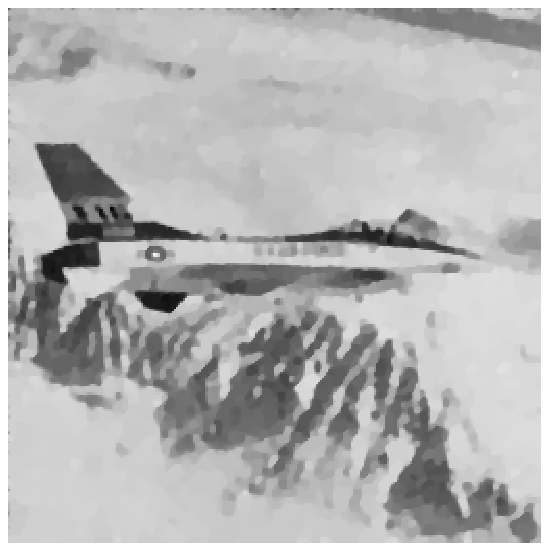}\\
(d) & (e) & (f)
\end{tabular}\end{center}
\caption{Image deblurring test problem. First row, impulse noise: (a) Original image ($512\times 512$ pixels); (b) Noisy blurred image, PSNR = 13.19; (c) Restored image, PSNR = 23.72. Second row, signal dependent Gaussian noise: (d) Original image ($256\times 256$ pixels); (e) Noisy blurred image, PSNR = 23.38; (f) Restored image, PSNR = 29.09}\label{fig:3}
\end{figure}
Panel (a) reports the values $(f(\xk)-f^*)/f^*$ provided by the algorithms \iPiano{}, \iPianoLA{}, VMILAn, and iISTA with respect to the computational time. 

\subsection{Image denoising and deblurring in presence of signal dependent Gaussian noise}\label{sec:second_test}

We now consider a second image deblurring problem, where we assume that the data contains signal dependent Gaussian noise. In this case the discrepancy functional is given by \cite{Chouzenoux-etal-2014}
\begin{equation}\nonumber
{\mathcal D}(Hx,g) = \frac 1 2 \sum_{i=1}^n\frac{((Hx)_i-g_i)^2}{a_i(Hx)_i+c_i}+\log(a_i(Hx)_i+c_i),
\end{equation} 
where $a_i,c_i$ are positive parameters. The functional above is nonconvex, smooth, and assuming that the entries of the blurring matrix $H$ are nonnegative, its domain contains the nonnegative orthant. Moreover, its gradient is Lipschitz continuous, even if an estimation of the Lipschitz constant is difficult to compute. As regularization term, we adopt the Total Variation function \cite{Rudin-Osher-Fatemi-1992}
\begin{equation}\nonumber
{\mathcal R}(x) = \rho\sum_{i=1}^n\|\nabla_ix\|,
\end{equation}
where the two components of $\nabla_i x\in\R^2$ contains the differences of the pixel $i$ with its vertical and horizontal neighbours, respectively. In this case, we can split the smooth and nonsmooth part of the objective function as $f = f_0+f_1$ with $f_0(x) = {\mathcal D}(Hx,g)$ and $f_1(x) =  {\mathcal R}(x) + \iota_{\geq0}(x)$.
For the minimization of the functional described above, we compare \iPiano{} and \iPianoLA{} with the variable metric forward backward algorithm VMFB proposed in \cite{Chouzenoux-etal-2014}, and again with iISTA. In particular, as a benchmark test, we adopt the same test problem provided in \cite{Logiciel}, with the same setting for the parameters $a_i,b_i, c_i, \rho$. The true image and the blurred noisy one are reported in Figure \ref{fig:3} (d) and (e).

As done in the previous section, we numerically compute an approximation of the optimal value $f^*$ {  by running all algorithms for a huge number of iterations, keeping the last function value for each algorithm, and then retaining the smallest value among them.} Then, we evaluate the optimization capability of the four algorithms in terms of the quantity $(f(\xk)-f^*)/f^*$ (see Figure \ref{fig:4} (b)). The parameters in \iPiano{}, \iPianoLA{} and iISTA are set in the same way as for the impulse noise test problem, while we run VMFB using the implementation released by the same authors \cite{Logiciel}.\\

The numerical results show that our proposed methods \iPiano{} and \iPianoLA{} are able to effectively solve challenging problems and that the presence of an inertial step can improve the effectiveness of standard forward--backward methods. On the one hand, \iPianoLA{} outperforms \iPiano{} on both test problems, which suggests that combining an inertial step with a linesearch along the descent direction (rather than along the arc) may be extremely beneficial in terms of computational times. On the other hand, \iPianoLA{} performs well also in comparison with VMILAn and VMFB, showing that inertial forward--backward algorithms can be competitive with variable metric approaches, provided that a sensible linesearch strategy is adopted.

\begin{figure}\begin{center}
\def\sclfct{0.4}
\begin{tabular}{cc}
\includegraphics[scale=\sclfct]{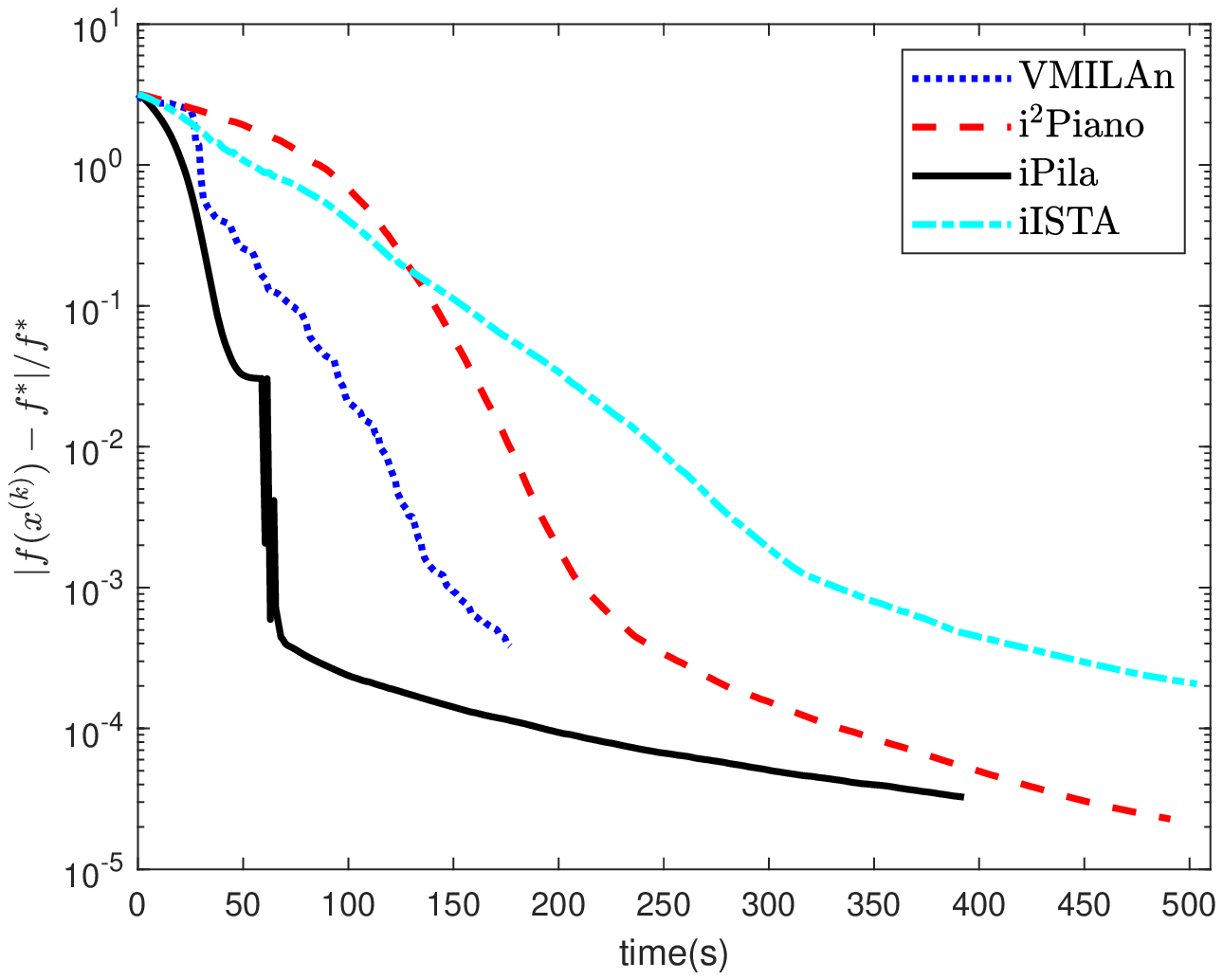}&\includegraphics[scale=\sclfct]{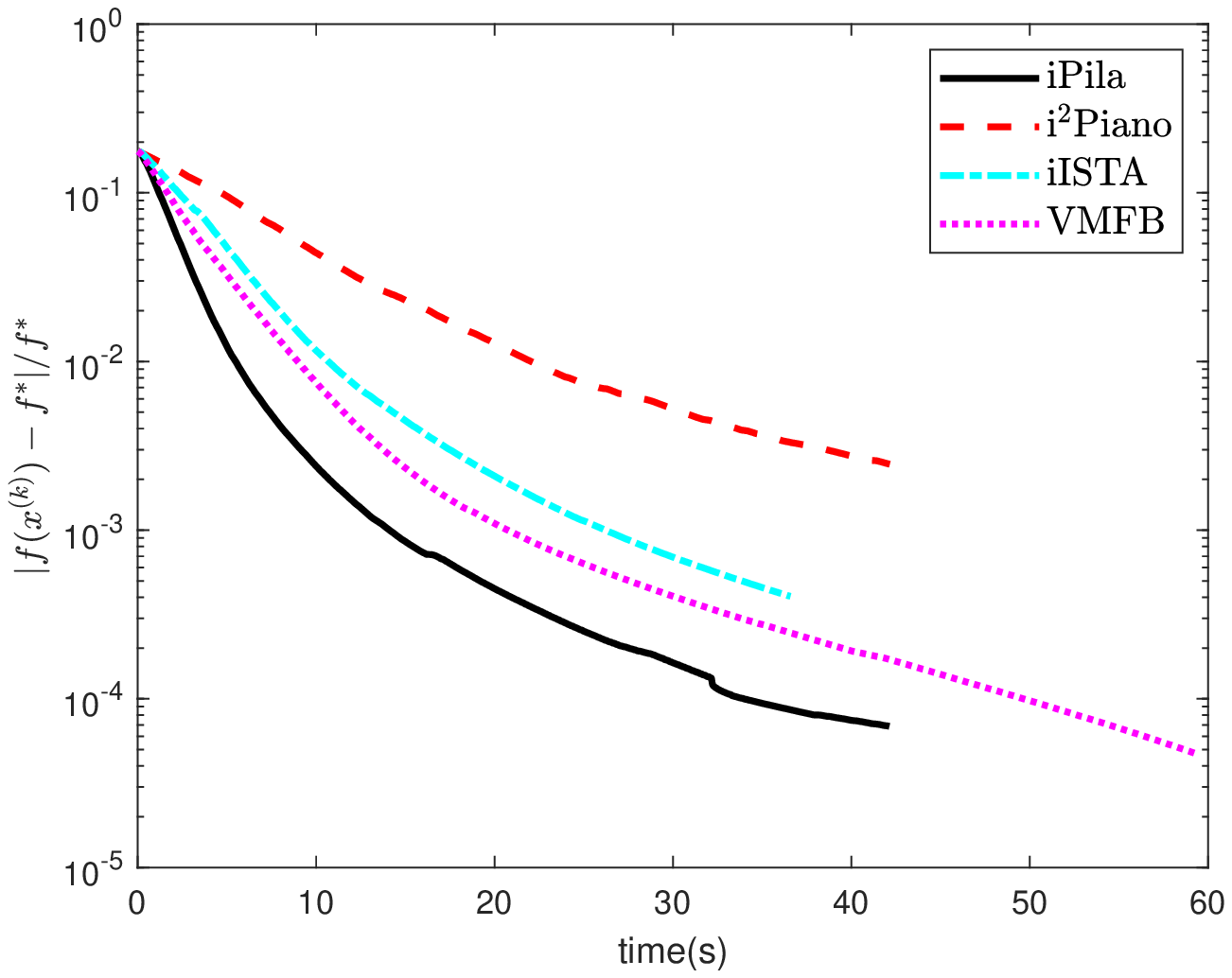}
\\
(a) & (b) 
\end{tabular}\end{center}
\caption{Relative decrease of the objective function with respect to the computational time for the two test problems. Panel (a): impulse noise. Panel (b): signal dependent Gaussian noise.}\label{fig:4}
\end{figure}

\section{Conclusions} 
We proposed two novel inertial-type forward--backward algorithms for solving nonsmooth nonconvex optimization problems. Both algorithms are equipped with implementable inexactness criteria for computing the proximal operator of the convex part. The first algorithm, i$^2$Piano, performs a classical backtracking procedure based on a local version of the Descent Lemma, whereas the second algorithm, iPila, is based on an innovative linesearch procedure along the descent direction of a suitable merit function. We showed that both algorithms converge to a stationary point of the problem, under some standard assumptions on the iterates sequence and the objective function. The convergence analysis is cast into an abstract framework that generalizes related work and thereby unfolds the formulation and convergence guarantees of our novel algorithms. We demonstrated the efficiency of the proposed algorithms on a couple of challenging image restoration problems, showing that the innovative approach employed in iPila may be beneficial in terms of computational times. Future work could be devoted to the design and analysis of novel effective rules for selecting the parameters in the iPila algorithm, in order to further improve its performance with respect to more traditional forward--backward algorithms. { Another possible development could be the adoption of a Bregman distance in the computation of the inexact proximal--gradient point of i2Piano and iPila, as done for other first-order methods in the KL framework \cite{Benning-et-al-21,Bolte-etal-2018,Bonettini-Loris-Porta-Prato-2016}.}

\appendix
\section{Appendix: convergence Analysis of \iPiano{} and \iPianoLA{}}\label{sec:convergence}
In this section we analyze the convergence properties of algorithms \iPiano{} and \iPianoLA{}, showing that they can be both considered as special cases of the abstract scheme presented in Section \ref{sec:abstract}. 
\subsection{Preliminary results}
We collect here new basic results concerning the inexactness criterion introduced in Section \ref{sec:inexact_inertial}, which is incorporated in our proposed algorithms. The following lemma is a consequence of the strong convexity of the function $h (\cdot;x,s)$ defined in \eqref{hsigma}, and will be often employed in the following. \silviacorr{In order to simplify the notation, here we omit the iteration index $k$.}
\begin{lemma} Suppose that Assumptions \ref{assi}--\ref{assii} hold true. For a given pair $(x,s)\in\dom(f_1)\times \R^n$, let $\hy,\ty$ be defined as in \eqref{hydef},\eqref{inexact_crit1}. Then, the following inequalities hold.
\begin{eqnarray}
\frac 1 {2\alpha} \|\hy-x\|^2 &\leq& \left(1+\frac \tau 2\right) (-h(\ty;x,s))\label{dist11}\\
\frac 1 {2\alpha} \|\ty-\hy\|^2&\leq& \frac{\tau}{2}  (-h(\ty;x,s))\label{dist22}\\
\frac \theta {2\alpha} \|\ty-x\|^2&\leq&  (-h(\ty;x,s)),\ \ \mbox{ with } \theta = 1/\left( \sqrt{1+\frac \tau 2}+\sqrt{\frac \tau 2}\right)^2\leq 1.\label{dist33}
\end{eqnarray}
\end{lemma}
\begin{proof}
Inequalities \eqref{dist11}--\eqref{dist22} follow by combining the strong convexity of the function $h(\cdot;x,s)$ and condition \eqref{inexact_crit1} as in \cite[Lemma 2]{Bonettini-Prato-Rebegoldi-2021}. As for \eqref{dist33}, we have
\begin{eqnarray*}
\frac{1}{2\alpha}\|\ty-x\|^2&=& \frac{1}{2\alpha}\|\ty-\hy + \hy-x\|^2 = \frac{1}{2\alpha}\|\ty-\hy\|^2 + \frac{1}{2\alpha}\|\hy-x\|^2 + \frac{1}{\alpha}\langle \ty-\hy,\hy-x\rangle\\
&\leq & \frac{1}{2\alpha}\|\ty-\hy\|^2 + \frac{1}{2\alpha}\|\hy-x\|^2 + \frac{1}{\alpha} \|\ty-\hy\|\cdot \|\hy-x\|\\
&\leq & \left(1+\frac \tau 2\right) (-h(\ty;x,s)) + \frac{\tau}{2}  (-h(\ty;x,s)) + 2\sqrt{1+\frac \tau 2}\sqrt{\frac \tau 2}(-h(\ty;x,s))\\
&=& \left( \sqrt{1+\frac \tau 2}+\sqrt{\frac \tau 2}\right)^2(-h(\ty;x,s)),
\end{eqnarray*}
where the last inequality follows from the application of \eqref{dist11}-\eqref{dist22}.
\end{proof}

The next lemma provides a subgradient $\hat{v}\in\partial f(\hat{y})$ whose norm is bounded from above by a quantity containing $\sqrt{ - h(\ty;x,s)} $. Its proof is omitted since it is almost identical to the one of Lemma 3 in \cite{Bonettini-Prato-Rebegoldi-2021}.
\begin{lemma} 
Suppose Assumptions \ref{assi}--\ref{assiii} hold true. Let $x$ be a point in $\dom(f_1)$ and let $\hy,\ty$ be defined as in \eqref{hydef}--\eqref{inexact_crit1}. Moreover, assume that $\alpha\in [\alpha_{min},\alpha_{max}]$, with $0<\alpha_{min}\leq\alpha_{max}$ and $\beta\in[0,\beta_{max}]$, with $\beta_{max}\geq 0$. Then, there exists a subgradient $\hat v\in\partial f(\hy)$ such that
\begin{eqnarray}
\|\hat v\| &\leq & p(\|\hy-x\| + \|x-s\|) \label{ine442}\\
&\leq& q(\sqrt{ - h(\ty;x,s)} + \|x-s\|) \label{ine44},
\end{eqnarray}
where the two constants $p,q$ depend only on $\alpha_{min},\alpha_{max},\beta_{max}$ and on the Lipschitz constant $L$.
\end{lemma}
%

The following lemma is the equivalent of Lemma 4-5 in \cite{Bonettini-Prato-Rebegoldi-2021}.

\begin{lemma}\label{lemma:crucialine1}
Suppose Assumptions \ref{assi}--\ref{assiii} hold true and assume $0<\alpha_{min}\leq\alpha\leq\alpha_{max}$, $\beta\in [0,\beta_{max}]$. Let $(x,s)$ be a point in $\dom(f_1)\times \R^n$ and let $\hy,\ty$ be defined as in \eqref{hydef}--\eqref{inexact_crit1}, for some $\tau\geq 0$. Then, there exists $c,d,\bar c,\bar d \in\R$ depending only on $\alpha_{min},\alpha_{max},\beta_{max},\tau$ such that
\begin{eqnarray}
f(\hy)&\geq& f(\ty) +ch(\ty;x,s) - d \|x-s\|^2\label{crucialine1}\\
f(\hy)&\leq & f(x) - \bar c h(\ty;x,s)+ \bar d\|x-s\|^2.\label{inef2}
\end{eqnarray}
\end{lemma}
\begin{proof}
From the Descent Lemma \cite[Proposition A.24]{Bertsekas-1999} we have
\begin{equation}\label{desclemma}
f_0(\hy)\geq f_0(\ty) - \langle \nabla f_0(\hy),\ty-\hy\rangle -\frac L 2 \|\ty-\hy\|^2.
\end{equation}
The inclusion $0\in \partial_\epsilon h(\ty;x,s)$
in \eqref{inexact_crit3} implies that there exists a vector $e\in \R^n$ with
\begin{equation}\label{def:ytilde3}
 \frac 1{2\alpha} \|e\|^2 \leq \epsilon 
\end{equation}
such that
\begin{equation}\nonumber
 -\frac 1 \alpha (\ty-x+\alpha \nabla f_0(x) -\beta(x-s)+ e)\in\partial_{\epsilon}f_1(\ty)
\end{equation}
(see \cite{Bonettini-Loris-Porta-Prato-Rebegoldi-2017} and references therein).
The definition of $\epsilon$-subdifferential implies
\begin{equation}\label{convexf1}
f_1(\hy)\geq f_1(\ty) -\frac 1 \alpha\langle \ty-x,\hy-\ty\rangle+\frac \beta \alpha\langle \hy-\ty,x-s\rangle-\langle\nabla f_0(x),\hy-\ty\rangle- \frac 1 \alpha \langle e,\hy-\ty\rangle - \epsilon.
\end{equation}
Summing inequalities \eqref{desclemma} and \eqref{convexf1} yields
\begin{eqnarray}
f(\hy)&\geq& f(\ty)-\langle \nabla f_0(x)- \nabla f_0(\hy),\hy-\ty\rangle +\frac \beta \alpha\langle \hy-\ty,x-s\rangle\nonumber\\
& & -\frac 1 \alpha \langle \ty-x,\hy-\ty\rangle- \frac 1 \alpha \langle e,\hy-\ty\rangle-\frac L 2 \|\ty-\hy\|^2 -\epsilon.\label{preine7}
\end{eqnarray}
Now we consider each term at the right-hand-side in the above inequality so as to obtain a lower bound. Using the Cauchy-Schwarz inequality, Assumption \ref{assiii}, \eqref{dist11} and \eqref{dist22} we obtain 
\begin{eqnarray}
\langle \nabla f_0(x)- \nabla f_0(\hy),\hy-\ty\rangle &\leq & \|\nabla f_0(x)- \nabla f_0(\hy)\|\|\hy-\ty\|\nonumber\\
&\leq& L\amax \sqrt{2\tau\left(1+\frac{\tau}{2}\right)}(-h(\ty;x,s)).  \label{ine5}
\end{eqnarray}
Similarly, using again the Cauchy-Schwarz inequality, \eqref{dist22} and \eqref{dist33}, we can write
\begin{eqnarray}
\frac 1 \alpha \langle \ty-x,\hy-\ty\rangle
&\leq& \frac 1{\amin}\|\ty-x\|\|\hy-\ty\|
\leq \frac{\amax}{\amin}\sqrt{\frac{2\tau}\theta}(-h(\ty;x,s)).\label{ine6}
\end{eqnarray}
Moreover, from \eqref{def:ytilde3} and \eqref{inexact_crit3} we obtain $\|e\|\leq \sqrt{2\amax\epsilon}\leq \sqrt{\amax\tau (-h(\ty;x,s))}$ which, using also \eqref{dist22}, yields
\begin{eqnarray}
\frac 1\alpha \langle e,\hy-\ty\rangle &\leq& \frac{1}{\alpha} \|e\|\|\hy-\ty\|
\leq \frac {\amax\tau} {\amin}(-h(\ty;x,s)).\label{ine7}
\end{eqnarray}
Finally, using \eqref{dist22}, we can also write
\begin{equation}
\frac\beta\alpha\langle \hy-\ty,x-s\rangle  \geq  - \frac{\beta}{2\alpha}(\|\hy-\ty\|^2 + \|x-s\|^2)\geq -\frac{\beta_{max}}{2\alpha_{min}}\left(-\alpha_{max}\tau h(\ty;x,s) + \|x-s\|^2\right).\label{ine8}
\end{equation}
Combining \eqref{preine7} with \eqref{dist33}, \eqref{ine5}, \eqref{ine6}, \eqref{ine7}, \eqref{ine8} and \eqref{inexact_crit3}, 
gives \eqref{crucialine1} with
\begin{equation*}
c = L\amax \sqrt{2\tau\left(1+\frac{\tau}{2}\right)}+\frac{\amax}{\amin}\sqrt{\frac{2\tau}\theta}+\frac {\amax\tau} {\amin}+\frac{L \amax\tau}{2}+\frac{\tau}{2}+\frac{\beta_{max}\amax\tau}{2\alpha_{min}}, \ \ d = \frac{\beta_{max}}{2\alpha_{min}}.
\end{equation*}
As for \eqref{inef2}, using the Descent Lemma we obtain
\begin{equation*}
f_0( y) \leq f_0(x) + \langle\nabla f_0(x),y - x\rangle + \frac L 2 \|y - x\|^2,
\end{equation*}
for all $x, y\in\dom(f_1)$.
Summing $f_1(y)$ on both sides yields
\begin{eqnarray*}
f(y) &\leq& f(x) + f_1(y)-f_1(x) +  \langle \nabla f_0(x),y - x\rangle + \frac L 2 \|y - x\|^2\\
    & \leq & f(x) + h(y;x,s) + \frac L 2 \|y - x\|^2+ \frac\beta\alpha\langle x-s,y - x\rangle\\
		&\leq& f(x) + h(y;x,s) + \frac L 2 \|y - x\|^2+ \frac\beta{2\alpha}(\|x-s\|^2 + \|y-x\|^2).
\end{eqnarray*}
From the previous inequality with $y=\hy$, recalling that $ h(\hy;x,s)\leq 0$ and combining with \eqref{dist11} yields \eqref{inef2}, where the constants are set as $\bar c = (L\amax+\beta_{max}) (1+{\tau}/{2})$, $\bar d = {\beta_{max}}/{(2\amin)}$.
\end{proof}

\subsection{Convergence analysis of \iPiano{}}\label{sec:conviPiano}

Our aim now is to frame \iPiano{} in the abstract scheme defined by Conditions \ref{definition:abstract} to enjoy the favourable convergence guarantees that are provided by Theorem \ref{thm:convergence}. \silviacorr{The line of the proof developed in this section is based on an extension of the arguments in \cite{Ochs-2019}.} We start the convergence analysis by showing that \iPiano{} is well-posed and that its parameters satisfy some useful relations.
\begin{lemma} 
The loop between {\textsc{STEP 2}} and {\textsc{STEP 6}} terminates in a finite number of steps. In particular, there exists $\overline{L}>0$ such that $L_k\leq \overline{L}$, $\forall \ k\geq 0$. Moreover, we have
\begin{equation}\label{betabound}
0\leq\beta_k\leq \frac{1+\theta\omega}{2}, \ \ \forall \ k\geq 0
\end{equation} 
and there exist two positive constants $\alpha_{min},\alpha_{max}$ with $0<\alpha_{min}\leq \alpha_{max}$ such that
$\alpha_k\in [\alpha_{min},\alpha_{max}] $, $\forall k\geq 0$. We also have
\begin{eqnarray}
\frac{1+\theta\omega}{2\alpha_k}-\frac{L_k}{2}-\frac{\beta_k}{2\alpha_k}&=&\delta\label{eq:relation}\\
\delta-\frac{\beta_k}{2\alpha_k}&=&\gamma.\label{eq:deltac}
\end{eqnarray}
\end{lemma}
\begin{proof}
Since $\eta>1$, after a finite number of steps the tentative value of $L_k$ satisfies $L_k\geq L$, where $L$ is the Lipschitz constant of $\nabla f_0$. Then, from the Descent Lemma, the inequality at {\textsc{Step 6}} is satisfied. From $\delta\geq \gamma$ we have $b_k\geq 1$, which implies \eqref{betabound}. A simple inspection shows that the following equalities hold:
\begin{equation}\nonumber
b_k=\frac{1+\theta\omega-{\beta_k}}{1+\theta\omega-2\beta_k}\Rightarrow 
\frac{L_k+2\delta}{L_k+2\gamma}=\frac{1+\theta\omega-{\beta_k}}{1+\theta\omega-2\beta_k},
\end{equation}
which leads to rewriting the parameter $\alpha_k$ as
\begin{equation}\label{eq:alter_alpha}
\alpha_k=\frac{1+\theta\omega-\beta_k}{L_k+2\delta}.
\end{equation}
Then there holds $\alpha_k\geq\alpha_{min}$ with $\alpha_{min} = (1+\theta\omega)/(2(\overline{L} + 2\delta))$ and, since $\theta\omega\leq 1$, we also have $\alpha_k\leq \alpha_{max}$ with $\alpha_{max} = 2/L_{min} $. 
Moreover, we have
\begin{equation}\nonumber
\frac{1+\theta\omega}{\alpha_k}-\frac{L_k}{2}-\frac{\beta_k}{2\alpha_k}=\frac{1+\theta\omega-\beta_k}{2\alpha_k}-\frac{L_k}{2}=\frac{L_k+2\delta}{2}-\frac{L_k}{2}=\delta
\end{equation}
and 
\begin{equation}\nonumber
\delta-\frac{\beta_k}{2\alpha_k}=\frac{1+\theta\omega}{2\alpha_k}-\frac{L_k}{2}-\frac{\beta_k}{\alpha_k}=\frac{1+\theta\omega-2\beta_k}{2\alpha_k}-\frac{L_k}{2}=\frac{L_k+2\gamma}{2}-\frac{L_k}{2}=\gamma.
\end{equation}
\end{proof}
Notice that the case $\tau=0$, which corresponds to the exact computation of the inertial proximal gradient point at \textsc{Step 5}, implies $\theta = 1$ and, choosing $\omega=1$, the parameters settings in \iPiano{} are exactly the same as in \cite{Ochs-etal-2014}. The need of introducing the parameter $\omega$ is mainly technical and will be explained in the following.\\ 
We now prove that condition \ref{H1} holds for \iPiano{} when the corresponding { merit} function $\Phi$ is defined as follows:
\begin{align}
&\Phi:\R^n\times \R^n \rightarrow \bR,\ \ \ \ \Phi(x,s) = f(x) + \delta\|x-s\|^2.\label{eq:Phi}
\end{align}

\begin{proposition}\label{prop:fund1}
Let $\{x^{(k)}\}_{k\in\mathbb{N}}$ be the sequence generated by \iPiano{}. Then there holds
\begin{equation}\label{H1iPiano_inexact}
\Phi(x^{(k+1)},x^{(k)}) \leq \Phi(x^{(k)},x^{(k-1)})  -\gamma\|\xk-\xkm\|^2+(1-\omega) h\k(\xkk;\xk,\xkm). 
\end{equation}
Consequently, there exist $\{s^{(k)}\}_{k\in\mathbb{N}}$, $\{d_k\}_{k\in\mathbb{N}}$, $\{a_k\}_{k\in\mathbb{N}}$ such that \ref{H1} holds with $\Phi$ as in \eqref{eq:Phi}.
\end{proposition}
\begin{proof}
By summing the quantity $f_1(\xkk)$ to both sides of inequality \eqref{eq:descent_inexact} we obtain
\begin{eqnarray}
f(\xkk) &\leq & f(\xk) \!+\! f_1(\xkk)-f_1(\xk) \!+\! \langle \nabla f_0(\xk),\xkk-\xk\rangle \!+\! \frac{L_k}{2}\|\xkk-\xk\|^2\nonumber\\
&=& f(\xk) + \hk(\xkk;\xk,\xkm) - \left(\frac 1{2\ak}-\frac{L_k}2\right)\|\xkk-\xk\|^2 \nonumber\\
& &+ \frac{\beta_k}{\alpha_k}\langle\xkk-\xk,\xk-\xkm\rangle\nonumber\\
&\leq & f(\xk) + h\k(\xkk;\xk,\xkm) - \left(\frac 1{2\ak}-\frac{L_k}2\right)\|\xkk-\xk\|^2\nonumber\\
& &  + \frac{\beta_k}{2\alpha_k}(\|\xkk-\xk\|^2+\|\xk-\xkm\|^2)\nonumber\\
&\leq & f(\xk)+(1-\omega)\hk(\xkk;\xk,\xkm) -\frac{\theta\omega}{2\alpha_k}\|\xkk-\xk\|^2 \nonumber\\
& &  - \left(\frac 1{2\ak}-\frac{L_k}2\right)\|\xkk-\xk\|^2+ \frac{\beta_k}{2\alpha_k}(\|\xkk-\xk\|^2+\|\xk-\xkm\|^2)\nonumber\\
&=& f(\xk) - \left(\frac {1+\theta\omega}{2\ak}-\frac{L_k}2-\frac{\beta_k}{2\alpha_k}\right)\|\xkk-\xk\|^2\nonumber\\
& &+(1-\omega)\hk(\xkk;\xk,\xkm) + \frac{\beta_k}{2\alpha_k}\|\xk-\xkm\|^2\nonumber
\end{eqnarray}
where the first equality is obtained by adding and subtracting to the right-hand-side the quantity $\|\xkk-\xk\|^2/(2\ak) +\beta_k/\alpha_k\langle \xkk-\xk,\xk-\xkm\rangle$, the subsequent inequality follows from the basic relation $2\langle a,b\rangle \leq \|a\|^2+\|b\|^2$ and the next one from \eqref{dist33}.\\
Recalling \eqref{eq:relation}, the above inequality can be conveniently rewritten as
\begin{align*}
f(x^{(k+1)})+\delta\|x^{(k+1)} &-x^{(k)}\|^2 \leq  f(x^{(k)})+\delta\|x^{(k)}-x^{(k-1)}\|^2\\
&\ \ +\left(\frac{\beta_k}{2\alpha_k}-\delta\right)\|x^{(k)}-x^{(k-1)}\|^2 +(1-\omega)h(\xkk;\xk,\xkm).
\end{align*}
Finally, exploiting \eqref{eq:deltac}, we obtain condition \ref{H1} with $\Phi$ given in \eqref{eq:Phi}, $a_k=1$ and
\begin{eqnarray}
s\k &= &\xkm \label{eq:sk}\\
d_k^2    &=& \gamma\|\xk-\xkm\|^2 -(1-\omega)\hk(\xkk;\xk,\xkm).\label{eq:dk}
\end{eqnarray}
\end{proof}
Under Assumption \ref{assiv}, $\Phi$ is bounded from below, hence, condition \ref{H1} implies 
\begin{equation}\label{fazzollo0}
\lim_{k\to\infty} \|x^{(k)}-x^{(k-1)}\| = 0
\end{equation}
and, if $\omega < 1$, also
\begin{equation}\label{fazzollo}
 \lim_{k\to\infty }\hk(\xkk;\xk,\xkm) = 0.
\end{equation}
The choice $\omega < 1$ is enforced when $\tau > 0$ in order to obtain \eqref{fazzollo}. If we take $\omega=1$, with the same arguments as above we still obtain \eqref{fazzollo0}. 
It is also worth noticing that, for large values of $\tau$, that is when a coarser accuracy is allowed in the computation of $\tyk$, the parameter $\theta$ can be very small and this also influences the choice of $\alpha_k$ and $\beta_k$.\\  

In order to prove condition \ref{H2}, let us now introduce the second { merit} function as follows 
\begin{equation}\label{def:F} 
{\mathcal F}:\R^n\times \R \to \bR,\ \ \ \ {\mathcal F}(u,\rho) = f(u) +\frac 1 2 \rho^2.
\end{equation}  
\def\yk{\hat x^{(k+1)}}
\def\tyk{\xkk}\noindent The following result is proved using similar arguments as the ones used in Lemma 4-5 in \cite{Bonettini-Prato-Rebegoldi-2021}.
\begin{proposition}\label{lemma:4new}
Let $\{\xk\}_\kinN$ be the sequence generated by \iPiano{} and, for each $k$, let the point $\yk$ be defined as 
\begin{equation}\label{ykdef}
\yk = \argmin_{y\in\R^n} \hk(y;\xk,\xkm).
\end{equation}
Then, there exist $\{\rho^{(k)}\}_{k\in\mathbb{N}}$, $\{r_k\}_{k\in\mathbb{N}}$ such that condition \ref{H2} holds with $\Phi$ defined in \eqref{eq:Phi}, $\mathcal{F}$ defined in \eqref{def:F}, $s^{(k)}$ defined in \eqref{eq:sk} and $u^{(k)}=\hat{x}^{(k+1)}$.
\end{proposition}
\begin{proof}
When $\omega=1$ we necessarily have $\tau = 0$, which means $\tyk = \hat{x}^{(k+1)}$. Therefore, \ref{H2} direcly follows from \ref{H1} with $u\k = \xkk$, $\rho^{(k)} = \sqrt{2\delta} \|\xk-\xkm\|$, $r_k=0$. Consider now the case $\omega <1$. From \eqref{crucialine1} and \eqref{inef2}, we directly obtain
\begin{eqnarray*}
f(\yk)&\geq& f(\tyk) + c\hk(\xkk;\xk,\xkm) - d\|\xk-\xkm\|^2\\
f(\yk)&\leq& f(\xk) - \bar c h(\tyk;\xk,\xkm)+ \bar d\|\xk-\xkm\|^2,
\end{eqnarray*}
where $c,d,\bar c,\bar d$ are defined as in Lemma \ref{lemma:crucialine1} and do not depend on $k$. 
Combining the two inequalities above we obtain
\begin{align*}
f(&\tyk) + \delta \|\xkk-\xk\|^2\leq \\
&\leq f(\yk) + \delta \|\xkk-\xk\|^2 { -c\hk(\xkk;\xk,\xkm)} + d\|\xk-\xkm\|^2\\
&\leq  f(\xk) - (c+\bar c) h(\tyk;\xk,\xkm)+ (d+\bar d)\|\xk-\xkm\|^2 + \delta  \|\xkk-\xk\|^2 .
\end{align*}
Recalling the definition of $\mathcal{F}$ in \eqref{def:F}, the above inequalities can be rewritten as
\begin{equation}\label{H21ok}
f(\tyk) + \delta \|\xkk-\xk\|^2\leq {\mathcal F}(\yk,\rho\k) \leq f(\xk) + \delta\|\xk-\xkm\|^2 + r_k
\end{equation}
where $\rho^{(k)}, r_k$ are given by
\begin{eqnarray}
\rho\k&=& \sqrt{2}(\delta\|\xkk-\xk\|^2 { -c\hk(\xkk;\xk,\xkm)} + d\|\xk-\xkm\|^2	)^{\frac 1 2} \nonumber\\
& &\label{sigmadef}\\
r_k &=& - (c+\bar c) h(\tyk;\xk,\xkm)+ (d+\bar d)\|\xk-\xkm\|^2+ \delta  \|\xkk-\xk\|^2.\nonumber
\end{eqnarray}
From \eqref{fazzollo0}--\eqref{fazzollo} we obtain $\displaystyle\lim_{k\to\infty} r_k = 0$, hence, assumption \ref{H2} is satisfied with the above settings and with $u\k = \yk$.
\end{proof}
Next we show that condition \ref{H3} holds for \iPiano{}. The following result combines elements of Lemma 3 in \cite{Bonettini-Prato-Rebegoldi-2021} and Lemma 17 in \cite{Ochs-2019}.
\begin{proposition}\label{lemma:3new}
There exist $b>0$, $I\subset \mathbb{Z}$, $\{\theta_i\}_{i\in I}$ such that condition \ref{H3} holds with $\{u^{(k)}\}_{k\in\mathbb{N}},\{\rho^{(k)}\}_{k\in\mathbb{N}},\{d_k\}_{k\in\mathbb{N}}$ defined as in Proposition \ref{lemma:4new} and in \eqref{eq:dk}, $\zeta_{k+1}=0$ and $b_{k+1}=1$.  
\end{proposition}
\begin{proof}
From the separable structure of ${\mathcal F}$, if $v\in\partial f(u)$ and $\rho\in\R$, we have that $(v,\rho)\in\partial {\mathcal F}(u,\rho)$. In particular, 
\begin{equation}\label{wend}
\|\partial {\mathcal F}(u,\rho)\|_- \leq \|v\| + \lvert \rho \rvert, \mbox{ for all } v\in\partial f(u), \rho\in\R.
\end{equation}
If $\omega = 1$, we are in the case $\tau = 0$. This means that $\xkk = \yk$, $\rho\k = \sqrt{2\delta} \|\xk-\xkm\|$. Moreover, \eqref{eq:dk} reduces to $d_k = \sqrt{\gamma}\|\xk-\xkm\| = \sqrt{\gamma/(2\delta)}\rho^{(k)}$. From \eqref{ine442}, we have that there exists a subgradient $\hat v^{(k+1)}\in\partial f(\yk)$ and a positive constant $p$ such that 
\begin{eqnarray*}
\|\hat v^{(k+1)}\| &\leq& p(\|\xkk-\xk\| + \|\xk-\xkm\|) = \frac{p}{\sqrt{\gamma}}(d_{k+1} + d_{k}).
\end{eqnarray*}
Hence, $\|\hat v^{(k+1)}\|+\rho\k\leq p d_{k+1}/{\sqrt{\gamma}} + d_k(p+\sqrt{2\delta})/{\sqrt{\gamma}} $ and, recalling \eqref{wend}, \ref{H3} follows with $b = (p+\sqrt{2\delta})/{\sqrt{\gamma}}$, $I=\{0,1\}$, $\theta_0=\theta_1=\frac 1 2$.\\
Consider now the case $\omega < 1$. From \eqref{ine44}, there exists a subgradient $\hat v^{(k+1)}\in\partial f(\yk)$ and a positive constant $q$ such that
\begin{eqnarray}
\|\hat v^{(k+1)}\| &\leq& q\sqrt{-h\k(\xkk;\xk,\xkm)} + q\|\xk-\xkm\|\nonumber\\
&\leq & q\sqrt{-h\k(\xkk;\xk,\xkm)} + q\sqrt{\frac{2\alpha_{max}}{\theta}}\sqrt{-h^{(k-1)}(\xk;\xkm,x^{(k-2)})}\nonumber\\
&\leq&\frac{q}{\sqrt{1-\omega}}\sqrt{-(1-\omega)h\k(\xkk;\xk,\xkm)+\gamma\|\xk-\xkm\|^2} \nonumber\\
& & + q\sqrt{\frac{2\alpha_{max}}{\theta(1-\omega)}}\sqrt{-(1-\omega)h^{(k-1)}(\xk;\xkm,x^{(k-2)})+\gamma\|\xkm-x^{(k-2)}\|^2} \nonumber\\
&\leq& \frac{q}{\sqrt{1-\omega}} d_k + q\sqrt{\frac{2\alpha_{max}}{\theta(1-\omega)}}d_{k-1} . \label{ine444} 
\end{eqnarray}
From \eqref{sigmadef} and \eqref{dist33} we have
\begin{eqnarray}
\rho\k &\leq&\sqrt{2}\left(-(2\delta\alpha_{max}/\theta+ c)\hk(\xkk;\xk,\xkm) + d\|\xk-\xkm\|^2	\right)^{\frac 1 2} \nonumber\\
&\leq& r d_k	\label{ine555}
\end{eqnarray}
where $r = \sqrt 2 \max \left\{(2\delta\alpha_{max}/\theta+c)/(1-\omega), d/\gamma \right\}^{\frac 1 2}$.
Then, combining \eqref{ine444} with \eqref{ine555}, in view of \eqref{wend} we obtain
\begin{equation*}
\|\partial {\mathcal F}(\yk,\rho\k)\|_-\leq \frac b 2(d_k+d_{k-1}),\ \ \mbox{ with }
b = 2\max\left\{ \frac{q}{\sqrt{1-\omega}}+r,q\sqrt{\frac{2\alpha_{max}}{\theta(1-\omega)}}\right\}
\end{equation*}
and the thesis follows with $I = \{1,2\}$, $\theta_1=\theta_2 = \frac 1 2$.
\end{proof}
The following lemma holds for all methods whose iterates satisfy \ref{H1},\ref{H2}, \ref{H3} with ${\mathcal F}$ defined as in \eqref{def:F}, and it is crucial to ensure condition \ref{H6} for \iPiano{}.
\begin{proposition}\label{prop:H4}
Let Assumptions \ref{assi}--\ref{assiv} be satisfied and assume that $\{\xk\}_\kinN$ satisfies \ref{H1},\ref{H2}, \ref{H3} with ${\mathcal F}$ defined as in \eqref{def:F}. If $\{(x^{(k_j)},\rkj)\}_{j\in\N}$ is a subsequence of $\{(\xk,\rk)\}_\kinN$ converging to some $ (x^*,\rho^*)\in\R^n\times \R^m$ such that $\lim_{j\to\infty}\| u^{(k_j)}-x^{(k_j)}\|=0$, then $\rho^*=0$ and $\lim_{j\to\infty}{\mathcal F}( u^{(k_j)},\rho^{(k_j)}) = {\mathcal F}(x^*,\rho^*)$.
\end{proposition}
\begin{proof}
Recalling that \ref{H1} implies $\displaystyle\lim_{k\to\infty} d_k=0$, from \ref{H3} and thanks to the separable structure of $\mathcal F$, we obtain that there exists $\hat v\k \in\partial f(u\k)$ such that $\displaystyle\lim_{k\to\infty} \|\hat v\k\|=\displaystyle\lim_{k\to\infty}\rk = 0$. In particular, in view of \eqref{sum:subdiff}, we can write $\hat v\k = \nabla f_0( u\k) + w\k$, where $w\k\in\partial f_1( u\k)$. Therefore, by continuity of $\nabla f_0$, the following implication holds
\begin{equation*}
\lim_{k\to\infty }\hat v\k=0\Rightarrow \lim_{j\to\infty} w^{(k_j)} = -\nabla f_0(x^*).
\end{equation*}
Adding the quantity $\frac 1 2 (\rho^{(k_j)})^2$ to both sides of the subgradient inequality yields
\begin{eqnarray*}
f_1(x^*) + \frac 1 2 (\rho^{(k_j)})^2&\geq & f_1( u^{(k_j)}) + \langle w^{(k_j)},x^*-u^{(k_j)}\rangle + \frac 1 2 (\rho^{(k_j)})^2\\
&=& {\mathcal F}( u^{(k_j)},\rho^{(k_j)}) -f_0(u^{(k_j)}) + \langle w^{(k_j)}, x^*-u^{(k_j)}\rangle
\end{eqnarray*}
where the last equality is obtained by adding and subtracting $f_0(u^{(k_j)})$ to the right-hand-side. Taking limits on both sides we obtain
\begin{equation*}
f_1(x^*) +\frac 1 2 (\rho^*)^2\geq \lim_{j\to\infty}{\mathcal F}(u^{(k_j)},\rho^{(k_j)}) -f_0(x^*),
\end{equation*}
which, rearranging terms, gives $\displaystyle\lim_{j\to\infty}{\mathcal F}(u^{(k_j)},\rho^{(k_j)}) \leq {\mathcal F}(x^*,\rho^*)$. On the other side, by assumption, $\mathcal F$ is lower semicontinuous, therefore $\displaystyle\lim_{j\to\infty}{\mathcal F}(u^{(k_j)},\rho^{(k_j)}) \geq {\mathcal F}(x^*,\rho^*)$, which completes the proof.
\end{proof}
We are now ready to prove Theorem \ref{thm:i2piano}, which states the convergence of the \iPiano{} iterates to a stationary point.

\begin{proof}[Proof of Theorem \ref{thm:i2piano}]
By Propositions \ref{prop:fund1}-\ref{lemma:4new}-\ref{lemma:3new}, we know that conditions \ref{H1}-\ref{H2}-\ref{H3} hold for \iPiano{}. Furthermore, if $\omega=1$, then $u\k = \hat x^{(k+1)}=\xkk$, and \eqref{fazzollo0} directly implies that $\displaystyle\lim_{k\to\infty} \|u\k -\xk\|= 0$. If $\omega<1$, using \eqref{dist11} and \eqref{fazzollo} we have
\begin{align*}
\lim_{k\to\infty} \|u\k -\xk\| &=\lim_{k\to\infty} \|\yk -\xk\| \\
&\leq  \lim_{k\to\infty}\sqrt{-2\alpha_{max}\left(1+\frac{\tau}{2}\right)\hk(\xkk;\xk,\xkm)}= 0.
\end{align*}
Then, in both cases, the assumptions of Proposition \ref{prop:H4} are satisfied and, therefore, condition \ref{H6} holds.
Finally, condition \ref{H4} follows from \eqref{eq:dk}, while condition \ref{H7} is trivially satisfied, since both sequences $\{a_k\}_{k\in\mathbb{N}}$ and $\{b_k\}_{k\in\mathbb{N}}$ are constant. Then, Theorem \ref{thm:convergence} applies and guarantees that the sequence $\{(x^{(k)},\rho^{(k)})\}_{k\in\mathbb{N}}$ converges to a stationary point $(x^*,\rho^*)$ of $\mathcal{F}$. Note that, since $\mathcal{F}$ is the sum of separable functions, its subdifferential can be written as $\partial \mathcal{F}(x,\rho)=\partial f(x)\times \{\rho\}$. Then, $(x^*,\rho^*)$ is stationary for $\mathcal{F}$ if and only if $\rho^*=0$ and $0\in\partial f(x^*)$. Hence, $x^*$ is a stationary point for $f$ and $\{x^{(k)}\}_{k\in\mathbb{N}}$ converges to it.  
\end{proof}

\subsection{Convergence analysis of \iPianoLA}\label{subsec:convergence_ipianola}
\def\tyk{\tilde y ^{(k)}}\noindent
This section aims to develop the convergence framework for \iPianoLA. The first issue to be addressed is the well posedness of the line--search algorithm. To this end, we prove first the following Lemma. 
\begin{lemma}\label{lemma:1_xs}
Let $d^{(k)}\in\mathbb{R}^{2n}$ be defined according to \eqref{dx_xs}-\eqref{ds_xs}, where $0< \ak\leq \alpha_{max}$, $\beta_k\geq 0$, $\gamma_k\geq\gamma_{min}$, for some given positive real constants $\alpha_{max}>0$, $\gamma_{min}\geq 0$. Define $\Delta_k \in\R_{\leq 0}$ as
\begin{equation}\label{defDeltak}
\Delta_k = h\k(\tyk;\xk,\sk) -\gamma_k\|\xk-\sk\|^2.
\end{equation}
Then, we have
\begin{align}
\Phi'(x^{(k)},s^{(k)};d_x^{(k)},d_s^{(k)})& \leq \Delta_k\label{inedelta}\\
&\leq - a\|\tyk-x^{(k)}\|^2-\gamma_{min} \|x^{(k)}-s^{(k)}\|^2\label{eq:Delta_k_ineq}
\end{align}
where $a>0$ is a positive constant. Therefore, $\Phi'(\xk,\sk;d_x\k,d_s\k)<0$ whenever $\tyk\neq \xk$ or $\xk\neq\sk$.
\end{lemma}
\begin{proof}
We first observe that \eqref{xs1} with $x=\xk$, $s=\sk$, $y=\tyk$, $d_x=d_x^{(k)}$, $d_s=d_s^{(k)}$, gives:
\begin{align}
\Phi'&(x^{(k)},s^{(k)};d_x^{(k)},d_s^{(k)})\nonumber\\
&\leq f_1(\tyk)-f_1(x^{(k)})+\langle \nabla f_0(x^{(k)})+x^{(k)}-s^{(k)},d_x^{(k)}\rangle +\langle s^{(k)}-x^{(k)},d_s^{(k)}\rangle\label{fc1}\\
 &= f_1(\tyk) -f_1(\xk) +\langle \nabla f_0(\xk),\tyk-\xk\rangle + \langle \xk-\sk,\tyk-\xk\rangle +\nonumber\\
 &\phantom{=} + \left(1+\frac{\beta_k}{\alpha_k}\right)\langle \sk-\xk,\tyk-\xk\rangle -\gamma_k\|\xk-\sk\|^2\nonumber\\
				&= f_1(\tyk) -f_1(\xk) +\langle \nabla f_0(\xk)-\frac{\beta_k}{\alpha_k}(\xk-\sk),\tyk-\xk\rangle-\gamma_k\|\xk-\sk\|^2\nonumber\\
				&\leq h\k(\tyk;\xk,\sk) -\gamma_k\|\xk-\sk\|^2 \label{fc2}\\
				&\leq  -\frac \theta {2\alpha_k}\|\tyk-\xk\|^2 -\gamma_k\|\xk-\sk\|^2\nonumber
\end{align}
where the last inequality follows from \eqref{dist33}.
Then, the thesis follows from $\alpha_k\leq \alpha_{max}$ with $a = \theta/{2\alpha_{max}}$.
\end{proof}

The well posedness of the line--search procedure and its main properties are summarized in the following lemma.
\begin{lemma}\label{lemma:arm}
Let the assumptions of Lemma \ref{lemma:1_xs} be satisfied with, in addition, $\beta_k\leq\beta_{max}$, $\gamma\leq \gamma_{max}$, ${ \beta_{max}\geq 0}$, $\gamma_{max}\geq 0$. Then, the Armijo backtracking line--search algorithm terminates in a finite number of steps, and there exists $\lambda_{min}>0$ such that the parameter $\lamk^+$ computed with the line--search algorithm satisfies
\begin{equation}\label{eq:lemmaarm2}
\lamk^+ \geq \lambda_{min}.
\end{equation} 
\end{lemma}
\begin{proof}
Let us first prove that 
\begin{equation}\label{eq:cor1_xs}
\Delta_k\leq - C \|\dk\|^2, \ \ \forall k\in\N.
\end{equation}
for a positive constant $C>0$. 
Setting $\delta_k = 1+\frac{\beta_k}{\alpha_k}$, the bounds on the parameters imply that $1\leq\delta_k\leq \bar\delta$, 
with $\bar \delta = 1+\frac{\beta_{max}}{\alpha_{min}}$.
By definition of $\dk$ in \eqref{dk_xs} we have
\begin{align}
\|\dk\|^2 &= \|d_x\k\|^2 + \|d_s\k\|^2\nonumber\\
&= \|\tyk-\xk\|^2 + \|\delta_k(\tyk-\xk) +\gamma_k(\xk-\sk)\|^2\nonumber\\
&= (1+\delta_k^2)\|\tyk-\xk\|^2 + \gamma_k^2\|\xk-\sk\|^2+2\delta_k\gamma_k\langle \tyk-\xk,\xk-\sk\rangle\nonumber\\
&\leq (1+\delta_k^2+\delta_k\gamma_k)\|\tyk-\xk\|^2 + (\gamma_k^2 + \delta_k\gamma_k)\|\xk-\sk\|^2\nonumber\\
&\leq (1+\bar \delta^2+\bar \delta \gamma_{max})\|\tyk-\xk\|^2 + \gamma_k(\gamma_{max} +  \bar\delta)\|\xk-\sk\|^2\nonumber\\
&\leq \frac 1 C (a\|\tyk-\xk\|^2 + \gamma_k\|\xk-\sk\|^2)\nonumber
\end{align}
where $C =1/\max\{(1+\bar \delta^2+\bar \delta \gamma_{max})/a,\gamma_{max} + \bar \delta\}$. Multiplying both sides of the last inequality above by $C$ and combining with \eqref{eq:Delta_k_ineq} gives \eqref{eq:cor1_xs}.\\
Since $\Phi_0$ has $M$-Lipschitz continuous gradient, with $M=L+2$ (see \eqref{nablaH0lip}), we can apply the Descent Lemma obtaining
\begin{align}
\Phi_0(\xk+\lambda d_x\k, \sk +\lambda d_s\k) &\leq \Phi_0(\xk,\sk) + \lambda\langle \nabla \Phi_0(\xk,\sk),\dk\rangle + \frac { L_{\Phi_0}} 2 \lambda^2\|\dk\|^2\nonumber\\
&\leq \Phi_0(\xk,\sk) + \lambda\langle \nabla \Phi_0(\xk,\sk),\dk\rangle - \frac { L_{\Phi_0}} {2C} \lambda^2\Delta_k,\label{desclemmaH0}
\end{align} 
where the second inequality follows from \eqref{eq:cor1_xs}. From the Jensen's inequality applied to the convex function $f_1$ we also obtain
\begin{eqnarray}
\Phi_1(\xk+\lambda d_x\k,\sk+\lambda d_s\k) &=& f_1(\xk+\lambda d_x\k)= f_1(\lambda \tyk+(1-\lambda)\xk)\nonumber\\
& \leq& (1-\lambda)f_1(\xk) + \lambda f_1(\tyk).\label{jensen}
\end{eqnarray}
Summing \eqref{desclemmaH0} with \eqref{jensen} gives
\begin{align*}
\Phi(\xk&+\lambda d_x\k, \sk +\lambda d_s\k) \leq \\
&\leq \Phi(\xk,\sk) + \lambda\left(f_1(\tyk) - f_1(\xk) +\langle \nabla \Phi_0(\xk,\sk),\dk\rangle\right)- \frac { L_{\Phi_0}} {2C} \lambda^2\Delta_k\\
& \leq  \Phi(\xk,\sk) + \lambda\Delta_k - \frac { L_{\Phi_0}} {2C} \lambda^2\Delta_k,
\end{align*}
where the last inequality follows from \eqref{fc1}--\eqref{fc2}. The above relation implies 
\begin{equation}\nonumber 
\Phi(\xk+\lambda d_x\k, \sk + \lambda d_s\k) \leq \Phi(\xk,\sk) + \lambda (1-\rho\lambda)\Delta_k,\ \ \forall\lambda\in[0,1]
\end{equation}
with $\rho = \frac { L_{\Phi_0}} {2C}$.
Moreover, comparing the above inequality with the Armijo condition \eqref{ine_arm} 
shows that the last one is surely fulfilled when $\lamk^+$ satisfies $ 1-\rho\lamk^+\geq \sigma$, that is when $\lamk^+\leq (1-\sigma)/\rho$.\\ Since $\lamk^+$ in the backtracking procedure is obtained starting from 1 and by successive reductions of a factor $\delta<1$, we have $\lamk^+\geq \delta^{ L_{\Phi_0}}$, where $ L_{\Phi_0}$ is  the smallest nonnegative integer such that $\delta^{ L_{\Phi_0}}\leq(1-\sigma)/\rho$. Therefore, \eqref{eq:lemmaarm2} is satisfied with $\lambda_{min} = \delta^{ L_{\Phi_0}}$.
\end{proof}



In the remaining of this section we show that \iPianoLA{} can be cast in the framework of the abstract scheme. We first show that conditions \ref{H1}--\ref{H3} are satisfied. 
\begin{proposition}\label{checkH1}
Let $\{(\xk,\sk)\}_\kinN$ be the sequence generated by \iPianoLA{}. Then, condition \ref{H1} holds with $d_k=\sqrt{-\Delta_k}$, $a_k = \sigma\lambda_{min}$. Moreover, under Assumption \ref{assiv}, we also have
\begin{equation}\label{tutto0}
0=\lim_{k\to\infty} \|\xk-\sk\| = \lim_{k\to\infty} h\k(\tyk;\xk,\sk) = \lim_{k\to\infty}\|\tyk-\xk\|.
\end{equation} 
\end{proposition}
\begin{proof}
From the updating rule at \textsc{STEP 7} and from Lemma \ref{lemma:arm}, we have
\begin{eqnarray*}
\Phi(\xkk,\skk) 
& \leq& \Phi(\xk,\sk) + \sigma \lamk\Delta_k\leq\Phi(\xk,\sk) + \sigma\lambda_{min}\Delta_k.
\end{eqnarray*}
Then, condition \ref{H1} is satisfied with $d_k^2=-\Delta_k$, $ a_k = \sigma\lambda_{min}$.
Since from Assumption \ref{assiv} $f$ is bounded from below, $\Phi$ is bounded from below as well. Therefore, \ref{H1} implies $\displaystyle-\sum_{k=0}^{\infty}\Delta_k<\infty$ which, in turn, yields $\displaystyle\lim_{k\to\infty}\Delta_k=0$. Recalling \eqref{eq:Delta_k_ineq}, this implies \eqref{tutto0}.
\end{proof}
In the following we describe the setup for proving \ref{H2}, with the second auxiliary function defined as in \eqref{def:F}.
\begin{proposition}\label{checkH2}
Let $\{(\xk,\sk)\}_{\kinN}$ be the sequence generated by Algorithm \iPianoLA{} with $\gamma_{min}>0$ and let ${\mathcal F}$ be defined as in \eqref{def:F}. Then, there exist $\{\rho^{(k)}\}_{k\in\mathbb{N}},\{r_k\}_\kinN\subset\R$, with $\lim_{k\to\infty}r_k=0$, such that \ref{H2} holds with $\uk=\hyk$, where $\hyk$ is the exact minimizer of $\hk(y;\xk,\sk)$, i.e.,
\begin{equation}\nonumber
\hyk = \argmin_{y\in\R^n} \hk(y;\xk,\sk).
\end{equation}
\end{proposition}
\begin{proof}
From {\textsc STEP 7}, we have
\begin{eqnarray*}\nonumber
\Phi(\xkk,\skk) &\leq& \Phi(\tyk,\xk) = f(\tyk) + \frac 1 2 \|\tyk-\xk\|^2\\
&\leq & f(\hyk) - \left(c+\frac \amax \theta\right)\hk(\tyk;\xk,\sk) + d\|\xk-\sk\|^2, 
\end{eqnarray*}
where the last inequality follows from \eqref{crucialine1} and \eqref{dist33}. Setting
\begin{equation}\label{rhok_xs}
\rho\k = \sqrt 2\left(-\left(c+\frac \amax \theta\right)\hk(\tyk;\xk,\sk)+ d\|\xk-\sk\|^2\right)^{\frac 1 2},
\end{equation}
we obtain $ \Phi(\xkk,\skk) \leq {\mathcal F}(\hyk,\rho\k)$, which represents the left-most inequality in \ref{H2}, with $u\k= \hyk$. On the other hand, from inequality \eqref{inef2} we obtain
\begin{equation*}
{\mathcal F}(\hyk,\rho\k) = f(\hyk) +\frac 1 2 (\rho\k)^2  \leq f(\xk) -\bar c \hk(\tyk;\xk,\sk) +\bar d\|\xk-\sk\|^2+\frac 1 2 (\rho\k)^2.
\end{equation*} 
Setting $r_k =  (\rho\k)^2/2  -\bar c \hk(\tyk;\xk,\sk) +(\bar d-\frac{1}{2})\|\xk-\sk\|^2
$, we can write
\begin{equation*}
{\mathcal F}(\hyk,\rho\k)  \leq f(\xk) +\frac 1 2 \|\xk-\sk\|^2 + r_k = \Phi(\xk,\sk) + r_k.
\end{equation*}
From \eqref{tutto0} we have that $\lim_{k\to\infty}r_k=0$ and this proves \ref{H2}.
\end{proof}
\begin{proposition}\label{checkH3}
Let $\{(\xk,\sk)\}_{\kinN}$ be the sequence generated by \iPianoLA{} with $\gamma_{min}>0$. Then, there exists a positive constant $b$ such that \ref{H3} is satisfied with $I = \{1\}$, $\theta_1 = 1$, $\zeta_k = 0$.
\end{proposition}
\begin{proof}
From \eqref{ine44} we know that there exists a subgradient $\hat v\k\in\partial f(\hat{y}^{(k)})$ such that
\begin{equation}\label{inevk_xs}
\|\hat v\k\| \leq q\sqrt{-\hk(\tyk;\xk,\sk)} + q\|\xk-\sk\|
\end{equation}
and, reasoning as in the proof of Proposition \ref{prop:H4}, it follows that
\begin{equation}\label{quinta}
\|\partial {\mathcal F}(\hyk,\rho\k)\|_- \leq \left\|\begin{pmatrix} \hat v\k\\ \rho\k\end{pmatrix} \right\|\leq \|\hat v\k\| + \lvert\rho\k\rvert.
\end{equation}
Let us analyze the two terms at the right-hand side of the inequality above, showing that both can be bounded from above with a multiple of $\sqrt{-\Delta_k}$. From \eqref{inevk_xs} we obtain
\begin{eqnarray*}\nonumber
\|\hat v\k\| &\leq& q\sqrt{-\hk(\tyk;\xk,\sk) + \gamma_{min}\|\xk-\sk\|^2}\\
& & + \frac{q}{\sqrt{\gamma_{min}}}\sqrt{\gamma_{min}\|\xk-\sk\|^2-\hk(\tyk;\xk,\sk)}.
\end{eqnarray*}
which, setting $A = q\left(1+1/\sqrt{\gamma_{min}}\right)$ and using \eqref{eq:Delta_k_ineq}, yields 
\begin{equation}\label{righe}
\|\hat v\k\| \leq A \sqrt{-\hk(\tyk;\xk,\sk) + \gamma_{min}\|\xk-\sk\|^2}\leq A\sqrt{-\Delta_k}.
\end{equation}
On the other hand, from definition \eqref{rhok_xs}
\begin{eqnarray*}
\rho\k 
&\leq & B \sqrt{-\hk(\tyk;\xk,\sk) + \gamma_k \|\xk-\sk\|^2} = B\sqrt{-\Delta_k}
\end{eqnarray*}
where $B = \sqrt 2\max\left\{c+  \amax /{\theta},  d / {\gamma_{min}}\right\}^{\frac 1 2}$.
Therefore, combining the last inequality above with \eqref{righe} and \eqref{quinta}, yields $\|\partial {\mathcal F}(\hyk,\rho\k)\|_-\leq (A+B)\sqrt{-\Delta_k}$ which proves that \ref{H3} is satisfied with $I = \{1\}$, $\theta_1 = 1$, $\zeta_k = 0$, $b=A+B$, $b_k=1$.
\end{proof}
We are now ready for proving Theorem \ref{thm:convipianola}, which states the main convergence result for Algorithm \iPianoLA{}.

\begin{proof}[Proof of Theorem \ref{thm:convipianola}]
By Proposition \ref{checkH1}, \ref{checkH2}, and \ref{checkH3}, we know that conditions \ref{H1}-\ref{H2}-\ref{H3} hold for \iPianoLA{}. 
From \eqref{dist11} we have $\displaystyle\lim_{k\to\infty}\|u\k-\xk\| = \displaystyle\lim_{k\to\infty}\|\hat{y}^{(k)}-\xk\|=0$. Hence we can apply Proposition \ref{prop:H4} and conclude that \ref{H6} holds.
Moreover, condition \ref{H4} holds as a consequence of Lemma \ref{lemma:1_xs}, since $d_k= \sqrt{-\Delta_k}$ and $\|\tilde{y}^{(k)}-x^{(k)}\|\geq \|x^{(k+1)}-x^{(k)}\|/\lambda_k\geq \|x^{(k+1)}-x^{(k)}\| $ (see STEP 7). Finally, condition \ref{H7} is trivially satisfied, since both sequences $\{a_k\}_{k\in\mathbb{N}}$ and $\{b_k\}_{k\in\mathbb{N}}$ are constant. Then Theorem \ref{thm:convergence} applies and guarantees that the sequence $\{(x^{(k)},\rho^{(k)})\}_{k\in\mathbb{N}}$ converges to a stationary point $(x^*,\rho^*)$ of $\mathcal{F}$. Note that, since $\mathcal{F}$ is the sum of separable functions, its subdifferential can be written as
$\partial \mathcal{F}(x,\rho)=\partial f(x)\times \{\rho\}$. Then, $(x^*,\rho^*)$ is stationary for $\mathcal{F}$ if and only if $\rho^*=0$ and $0\in\partial f(x^*)$. Hence $x^*$ is a stationary point for $f$ and $\{x^{(k)}\}_{k\in\mathbb{N}}$ converges to it.  
\end{proof}

\paragraph{Funding and data availability statement}Silvia Bonettini, Marco Prato and Simone Rebegoldi are members of the INdAM research group GNCS. Peter Ochs acknowledges funding by the German Research Foundation (DFG Grant OC 150/3-1).\\
The test image \texttt{barbara} in the experiments described in Section \ref{sec:first_test} is included in the software available at \cite{iPila2021}, while the \texttt{jetplane} one used in Section \ref{sec:second_test} can be downloaded from \cite{Logiciel}.

\bibliography{biblio_Silvia}
\bibliographystyle{plain}

\end{document}